\documentclass[a4paper,11pt,reqno]{article}
\usepackage{latexsym, amsmath, amsfonts, amssymb, amsthm, amscd, epsfig}
\usepackage{url}

\usepackage[font=small, labelfont={sc}]{caption}

\usepackage[a4paper,scale={0.73,0.75},marginratio={1:1},footskip=10mm,headsep=10mm]{geometry}

\usepackage{hyperref}

\usepackage{color}
\definecolor{mycolorred}{rgb}{1, 0, 0}

\numberwithin{equation}{section}

\def \beq {\begin{eqnarray}}
\def \eeq {\end{eqnarray}}
\def \beqn {\begin{eqnarray*}}
\def \eeqn {\end{eqnarray*}}

\newtheorem{theorem}{Theorem}[section]
\newtheorem{itlemma}[theorem]{Lemma}
\newtheorem{itproposition}[theorem]{Proposition}
\newtheorem{itcorollary}[theorem]{Corollary}
\newtheorem{itremark}[theorem]{Remark}
\newtheorem{itdefinition}[theorem]{Definition}
\newtheorem{itexample}[theorem]{Example}
\newtheorem{itclaim}[theorem]{Claim}
\newtheorem{itfact}[theorem]{Fact}
\newtheorem{itassumption}[theorem]{Assumption}

\newenvironment{fact}{\begin{itfact}\rm}{\end{itfact}}
\newenvironment{claim}{\begin{itclaim}\rm}{\end{itclaim}}
\newenvironment{lemma}{\begin{itlemma}}{\end{itlemma}}
\newenvironment{remark}{\begin{itremark}\rm}{\end{itremark}}
\newenvironment{corollary}{\begin{itcorollary}}{\end{itcorollary}}
\newenvironment{proposition}{\begin{itproposition}}{\end{itproposition}}
\newenvironment{definition}{\begin{itdefinition}\rm}{\end{itdefinition}}
\newenvironment{example}{\begin{itexample}\rm}{\end{itexample}}
\newenvironment{assumption}{\begin{itassumption}}{\end{itassumption}}

\newcommand{\be}[1]{\begin{equation}\label{#1}}
\newcommand{\ee}{\end{equation}}
\newcommand{\bl}[1]{\begin{lemma}\label{#1}}
\newcommand{\br}[1]{\begin{remark}\label{#1}}
\newcommand{\brs}[1]{\begin{remarks}\label{#1}}
\newcommand{\bt}[1]{\begin{theorem}\label{#1}}
\newcommand{\bd}[1]{\begin{definition}\label{#1}}
\newcommand{\bp}[1]{\begin{proposition}\label{#1}}
\newcommand{\bc}[1]{\begin{corollary}\label{#1}}
\newcommand{\bfact}[1]{\begin{fact}\label{#1}.}
\newcommand{\bex}[1]{\begin{example}\label{#1}.}
\newcommand{\ec}{\end{corollary}}
\newcommand{\efact}{\end{fact}}
\newcommand{\eex}{\end{example}}
\newcommand{\el}{\end{lemma}}
\newcommand{\er}{\end{remark}}
\newcommand{\ers}{\end{remarks}}
\newcommand{\et}{\end{theorem}}
\newcommand{\ed}{\end{definition}}
\newcommand{\ep}{\end{proposition}}
\newcommand{\epr}{\end{proof}}
\newcommand{\bpr}{\begin{proof}}
\newcommand{\bcl}[1]{\begin{claim}\label{#1}}
\newcommand{\ecl}{\end{claim}}
\newcommand{\bas}[1]{\begin{assumption}\label{#1}}
\newcommand{\eas}{\end{assumption}}

\newcommand{\ecs}{\end{corollary}}
\newcommand{\eers}{\end{exercise}}
\newcommand{\eexs}{\end{example}}
\newcommand{\eems}{\end{example}}
\newcommand{\els}{\end{lemma}}
\newcommand{\eles}{\end{lemmaex}}
\newcommand{\ets}{\end{theorem}}
\newcommand{\eds}{\end{definition}}
\newcommand{\eps}{\end{proposition}}

\newcommand{\bi}{\begin{itemize}}
\newcommand{\ei}{\end{itemize}}
\newcommand{\ben}{\begin{enumerate}}
\newcommand{\een}{\end{enumerate}}

\def\vbar{\mathchoice{\vrule height6.3ptdepth-.5ptwidth.8pt\kern-.8pt}
   {\vrule height6.3ptdepth-.5ptwidth.8pt\kern-.8pt}
   {\vrule height4.1ptdepth-.35ptwidth.6pt\kern-.6pt}
   {\vrule height3.1ptdepth-.25ptwidth.5pt\kern-.5pt}}
\def\fudge{\mathchoice{}{}{\mkern.5mu}{\mkern.8mu}}
\def\bbc#1#2{{\rm \mkern#2mu\vbar\mkern-#2mu#1}}
\def\bbb#1{{\rm I\mkern-3.5mu #1}}
\def\bba#1#2{{\rm #1\mkern-#2mu\fudge #1}}
\def\bb#1{{\count4=`#1 \advance\count4by-64 \ifcase\count4\or\bba A{11.5}\or
   \bbb B\or\bbc C{5}\or\bbb D\or\bbb E\or\bbb F \or\bbc G{5}\or\bbb H\or
   \bbb I\or\bbc J{3}\or\bbb K\or\bbb L \or\bbb M\or\bbb N\or\bbc O{5} \or
   \bbb P\or\bbc Q{5}\or\bbb R\or\bbc S{4.2}\or\bba T{10.5}\or\bbc U{5}\or
   \bba V{12}\or\bba W{16.5}\or\bba X{11}\or\bba Y{11.7}\or\bba Z{7.5}\fi}}

\makeatletter
\newtheorem*{rep@theorem}{\rep@title} \newcommand{\newreptheorem}[2]{%
\newenvironment{rep#1}[1]{%
\def\rep@title{\bf #2 \ref{##1} }%
\begin{rep@theorem} }%
{\end{rep@theorem} } }
\makeatother
\newreptheorem{theorem}{Theorem}
\newreptheorem{lemma}{Lemma}

\def\dim{\mathrm{dim}}

\def \R {{\mathbb R}}
\def \P {{\mathbb P}}

\def \C {{\cal{C}}}

\def \N {{\mathbb N}}
\def \PR {{\mathbb P}}
\def \E {{\mathbb E}}

\def \o {\omega}

\def\<{\langle}
\def\>{\rangle}

\def \DD {{\mathbb D}}
\def \s {y}

\def \H {{\cal{H}}}

\newcommand{\ba}[1]{\addtocounter{for}{1} \begin{eqnarray}\label{#1}}
\newcommand{\ea}{\end{eqnarray}}

\def\sqr#1#2{{\vcenter{\vbox{\hrule height .#2pt
                             \hbox{\vrule width .#2pt height#1pt \kern#1pt
                                   \vrule width .#2pt}
                             \hrule height .#2pt}}}}

\def\pmb#1{\setbox0=\hbox{#1}%
   \kern-.025em\copy0\kern-\wd0
   \kern.05em\copy0\kern-\wd0
   \kern-.025em\raise.0433em\box0 }
\def\sqr#1#2{{\vcenter{\vbox{\hrule height.#2pt
     \hbox{\vrule width.#2pt height#1pt \kern#1pt
   \vrule width.#2pt}\hrule height.#2pt}}}}

\def\ve{\varepsilon}
\def\s{\sigma}
\def\d{\delta}
\def\l{\lambda}

\def\g{\gamma}
\def\G{\Gamma}
\def\a{\alpha}
\def\th{\theta}

\def\r{\rho}

\def\Th{\Theta}

\def\cal{\mathcal}

\newenvironment{myenumerate}{%
\begin{list}{\labelenumi}
	{%
	\setlength{\itemsep}{0.4em}%
	\setlength{\topsep}{0.5em}%
	\setlength\leftmargin{2.6em}%
	\setlength\labelwidth{2.15em}%
	\setlength{\labelsep}{0.45em}%
	\usecounter{enumi}%
	}%
	}%
{\end{list}
}

\renewenvironment{enumerate}{
\renewcommand{\theenumi}{\arabic{enumi}}%
\renewcommand{\labelenumi}{{\rm(\theenumi)}}%
\begin{myenumerate}}%
{\end{myenumerate}}

{\end{myenumerate}}

{\end{myenumerate}}

\newenvironment{myitemize}{%
\begin{list}{$\bullet$}%
 	{%
	\setlength{\itemsep}{0.4em}%
	\setlength{\topsep}{0.5em}%
	\setlength\leftmargin{2.6em}%
	\setlength\labelwidth{2.15em}%
	\setlength{\labelsep}{0.45em}%
	}%
	}%
{\end{list}}

\renewenvironment{itemize}{
\begin{myitemize}}%
{\end{myitemize}}

\begin{document}

\parindent 0pt

\title{Diffusions under a local strong H\"ormander condition.\\Part I: density estimates}
\author{ \textsc{Vlad Bally}\thanks{%
Universit\'e Paris-Est, LAMA (UMR CNRS, UPEMLV, UPEC), MathRisk INRIA, F-77454
Marne-la-Vall\'{e}e, France. Email: \texttt{bally@univ-mlv.fr} }\smallskip \\
\textsc{Lucia Caramellino}\thanks{%
Dipartimento di Matematica, Universit\`a di Roma - Tor Vergata, Via della
Ricerca Scientifica 1, I-00133 Roma, Italy. Email: \texttt{%
caramell@mat.uniroma2.it}. }\smallskip\\
\textsc{Paolo Pigato}\thanks{%
INRIA, Villers-l\`es-Nancy, F-54600, France
Universit\'e de Lorraine, IECL, UMR 7502, Vandoeuvre-l\`es-Nancy, F-54600, France \texttt{%
paolo.pigato@inria.fr}. }\smallskip\\
}
\maketitle

\begin{abstract}
We study lower and upper bounds for the density of a diffusion process in $\R^n$ in a small (but not asymptotic) time, say $\delta$. We assume that the diffusion coefficients $\sigma_1,\ldots,\sigma_d$ may degenerate at the starting time $0$ and point $x_0$ but they satisfy a strong H\"ormander condition involving the first order Lie brackets. The density estimates are written in terms of a norm which accounts for the non-isotropic structure of the problem: in a small time $\delta$, the diffusion process propagates with speed $\sqrt{\delta}$ in the direction of the diffusion vector fields $\sigma _{j}$ and with speed $\delta=\sqrt{\delta}\times \sqrt{\delta}$ in the direction of $[\sigma _{i},\sigma _{j}]$.
In the second part of this paper, such estimates will be used in order to study lower and upper bounds for the probability that the diffusion process remains in a tube around a skeleton path up to a fixed time.
\end{abstract}

\tableofcontents

\date{}

\section{Introduction}
In this paper we study bounds for the density of a diffusion process at a small time under a local strong H\"ormander condition. To be more precise, let $X$ denote the process in $\R^n$ solution to
\be{equation-intro}
dX_{t}=\sum_{j=1}^{d}\s
_{j}(t,X_{t})\circ dW_{t}^{j}+b(t,X_{t})dt,\quad \quad X_0=x_0.
\ee
where $W=(W^{1},...,W^{d})$ is a standard Brownian motion and $\circ
dW_{t}^{j} $ denotes the Stratonovich integral.
We assume nice differentiability and boundedness or sublinearity properties for the diffusion coefficients  $b$ and $\sigma_j$, $j=1,\ldots,d$, and we consider a degenerate case:
\begin{equation}\label{dim}
\dim\,\sigma(0,x_0)=\dim\,\mathrm{Span}\{\s_1(0,x_0), \ldots,\s_d(0,x_0)\}<n,
\end{equation}
$\dim\,S$ denoting the dimension of the vector space $S$.
Our aim is to study lower and upper bounds for the density of the solution to
\eqref{equation-intro} at a small (but not asymptotic) time, say $\delta$, under the following local strong H\"ormander condition:
\begin{equation}\label{hor-intro}
\mathrm{Span}\{\sigma_i(0,x_0),[\sigma_p,\sigma_j](0,x_0),\ i,p,j=1,\ldots,d\}=\R^n
\end{equation}
in which $[\cdot,\cdot]$ denotes the standard Lie bracket vector field. Notice that we ask for a H\"orman\-der condition at time $0$. Our estimates are written in terms of a norm which reflects the
non-isotropic structure of the problem: roughly speaking, in a small time interval of length $\delta$, the diffusion process moves with speed $\sqrt{\delta}$ in the direction of the diffusion vector fields $\sigma _{j}$ and with speed $\delta=\sqrt{\delta}\times \sqrt{\delta}$ in the direction of $[\sigma _{i},\sigma _{j}]$.
In order to catch this behavior we introduce
the following norms. Let $A_\delta(0,x_0)$ denote the $n\times d^2$ matrix
$$
A_\delta(0,x_0)=[A_{1,\delta}(0,x_0),\ldots,A_{d^2,\delta}(0,x_0)]
$$
where the general column $A_{l,\delta}(0,x_0)$, $l=1,\ldots,d^2$, is defined as follows:
\begin{itemize}
\item
for $l=(p-1)d+i$ with  $p,i\in\{1,\dots,d\}$ and $p\neq i$ then
$$
A_{l,\delta}(0,x_0) =[\sqrt\delta\s_{i},\sqrt\delta\s_{p}](0,x_0)
=\delta[\s_{i},\s_{p}](0,x_0);
$$
\item
for $l=(p-1)d+i$ with  $p,i\in\{1,\dots,d\}$ and $p= i$ then
$$
A_{l,\delta}(0,x_0) =\sqrt\delta\,\s_{i}(0,x_0).
$$
\end{itemize}
Under \eqref{hor-intro}, the rank of $A_\delta(0,x_0)$ is equal to $n$, hence the following norm is well defined:
$$
|\xi|_{A_\delta(0,x_0)}=\<(A_\delta(0,x_0)A_\delta(0,x_0)^T)^{-1}\xi,\xi\>^{1/2},\quad \xi\in\R^n,
$$
where the supscript $T$ denotes the transpose and $\<\cdot,\cdot\>$ stands for the standard scalar product. We prove in \cite{BCP2} that the metric
given by this norm is locally equivalent with the control distance $d_{c}$
(the Carath\'eodory distance) which is usually used in this framework. We denote by $p_{\delta}(x_0,\cdot)$ the density of the solution to \eqref{equation-intro} at time $\delta$.  Under \eqref{hor-intro} and assuming suitable hypotheses on the boundedness and sublinearity of the coefficients $b$ and $\sigma_j$, $j=1,\ldots, d$ (see Assumption \ref{assumption1} for details), we prove the following result (recall $\dim\,\sigma(0,x_0)$ given in \eqref{dim}):
\begin{description}
\item[]
\textsc{[lower bound]}
there exist positive constants $r,\d_*,C$ such that for every $\d\leq \d_*$ and for every $y$ with $|y- x_0-b(0,x_0)\d|_{A_\d(0,x_0)}\leq r$ one has
\[
p_\delta(x_0,y)\geq \frac{1}{C \d^{n-\frac{\dim\,\sigma(0,x_0)}{2} }};
\]
\item []
\textsc{[upper bound]}
for any $p>1$, there exists a positive constant $C$ such that for every $\d\leq 1$ and for every $y\in \R^n$ one has
\[
p_\delta(x_0,y)\leq  \frac{1}{\d^{n-\frac{\dim\,\sigma(0,x_0)}{2}}
} \frac{C}{1+|y-x_0|_{A_\d(0,x_0)}^p}.
\]
\end{description}
This is stated in Theorem \ref{main-th}, where an exponential upper bound is achieved as well, provided that stronger boundedness assumptions on the diffusion coefficients hold (see Assumption \ref{assumption1bis}).

In the context of a degenerate diffusion coefficient which fulfills a strong H\"ormander condition, the main result in this direction is due to Kusuoka and Stroock. In the celebrated paper \cite{KusuokaStroock:87}, they prove the following two-sided Gaussian bounds: there exists a constant $M\geq 1$ such that
\be{ksintro}
\begin{split}
&\frac{1}{M |B_{d_c}(x_0,\delta^{1/2})|}\exp\left( -\frac{M d_c(x_0,y)^2}{\delta}\right)\\
&\quad \quad \quad \leq p_\delta(x_0,y)
\leq \frac{M}{|B_{d_c}(x_0,\delta^{1/2})|}\exp\left( -\frac{d_c(x_0,y)^2}{M \delta}\right)
\end{split}
\ee
where $\delta\in (0,1]$, $x_0,y\in \R^n$, $B_{d_c}(x,r)=\{y\in \R^n: d_c(x,y)<r\}$, $d_c$ denoting the control (Carath\'eodory) distance, and $|B_{d_c}(x,r)|$ stands for the Lebesgue measure of $B_{d_c}(x,r)$. It is worth to be said that \eqref{ksintro} holds under special hypotheses: in \cite{KusuokaStroock:87} it is assumed that the coefficients do not depend on the time variable and that $b(x)=\sum_{j=1}^d \a_i \s_i(x)$, with $\a_i\in C_b^\infty(\R^n)$ (i.e. the drift is generated by the vector fields of the diffusive part, which is a quite restrictive hypothesis).
 Other
celebrated estimates for the heat kernel under strong H\"{o}rmander
condition are provided in \cite{BenArousLeandre:91, BenArousLeandreII:91}.
The subject has also been widely studied by analytical methods - see
for example \cite{Jerison1987} and \cite{Sanchez-Calle1984}.  We stress that these are asymptotic results, whereas we prove estimate for a finite, positive and fixed time. In \cite{pigato:14}, \cite{DelarueMenozzi:10}, non-isotropic norms similar to $|\cdot|_{A_\d(0,x_0)}$ are used to
provide density estimates for SDEs under H\"ormander conditions of weak type. We also refer to \cite{CattiauxMesnager:02}, which considers the existence of the density for SDEs with time dependent coefficients, under very weak regularity assumption.

The paper \cite{BCP2} will follow the present work, considering related results from a ``control'' point of view, discussing in particular tube estimates and the connection with the control-Carath\'eodory distance.
Tube estimates are estimates on the probability that an It\^o process remains around a deterministic path up to a given time. These will be obtained from a concatenation of the short-time density estimates presented here.
Then we will consider, as in \cite{KusuokaStroock:87}, $b(t,x)=b(x)$ and $\s(t,x)=\s(x)$. Defining the semi distance $d$ via: $d(x,y) < \sqrt{\delta}$ if and only if $|x-y|_{A_\delta(x)}<1$, we will prove in \cite{BCP2} the local equivalence of $d$ and $d_c$. This will give a rewriting of the upper/lower estimates of the density in terms of the control distance as well.

The paper is organized as follows. In Section \ref{settings} we set-up the framework and give the precise statement of our main result (Theorem \ref{main-th}). The proof is split in two sections: Section \ref{lower}, which is devoted to the lower bound, and Section \ref{upper}, in which we deal with the upper bound. The main tool we are going to use is given by the estimates of localized densities which have been developed in \cite{BC14}.  These need to use techniques from Malliavin calculus, so we briefly report in Appendix \ref{app-mall} all these arguments. But in order to set-up our program we also require some other facts, which have been collected in other appendices. First, we use a key decomposition of the solution $X_\delta$ to \eqref{equation-intro} at a small (but not asymptotic) time $\delta$ (see Section \ref{lower-key}), and we postpone the proof in Appendix \ref{app-dec}. This decomposition allows us to work with a random variable whose law, conditional to a suitable $\sigma$-algebra, is Gaussian, and in Appendix \ref{app-supp} we study some useful support properties that are applied to our case. Moreover, since the key-decomposition brings to handle a perturbed Gaussian random variable, in Appendix \ref{sectioninversefunction} we prove density estimates via local inversion for such kind of random variables.

\section{Notations and main results}\label{settings}

We need to recall some notations. For $f,g:\R^{+}\times \R^{n}\rightarrow \R^{n}$ we define the directional derivative (w.r.t. the space variable $x$) $\partial_{g}f(t,x)=\sum_{i=1}^{n}g^{i}(t,x)\partial _{x_{i}}f(t,x)$, and  we recall that the Lie bracket (again w.r.t. the space variable)
is defined as $[g,f](t,x)=\partial _{g}f(t,x)-\partial _{f}g(t,x)$.
Let $M\in \mathcal{M}_{n\times m}$ be a matrix with full row rank. We write $M^T$ for the transposed matrix, and $MM^T$ is invertible. We denote by $\l_*(M)$ (respectively $\l^*(M))$ the smallest (respectively the largest) singular value of $M$.
We recall that singular values are the square roots of the eigenvalues of $MM^T$, and that, when $M$ is symmetric, singular values coincide with the absolute values of the eigenvalues of $M$. In particular, when $M$ is a covariance matrix, $\l_*(M)$ and $\l^*(M)$ are the smallest and the largest eigenvalues of $M$.

We consider the following norm on $\R^{n}$:
\be{Not4}
\left\vert y\right\vert _{M}=\sqrt{\left\langle (MM^T)^{-1}y,y\right\rangle}.
\ee
Hereafter, $\alpha =(\alpha _{1},...,\alpha _{k})\in \{1,...,n\}^{k}$ represents a multi-index with length $\left\vert \alpha \right\vert =k$ and $%
\partial _{x}^{\alpha }=\partial _{x_{\alpha _{1}}}...\partial _{x_{\alpha
_{k}}}$. We allow the case $k=0$, giving $\alpha=\emptyset$, and $\partial_x^\alpha f= f$.
Finally, for given vectors $v_1,\ldots,v_n\in \R^m$, we define $\langle v_1,\ldots,v_n\rangle \subset \R^{m}$ the vector space spanned by $v_1,\ldots,v_n$.

Let $X$ denote the process in $\R^n$ already introduced in \eqref{equation-intro}, that is
\be{equation}
dX_{t}=\sum_{j=1}^{d}\s
_{j}(t,X_{t})\circ dW_{t}^{j}+b(t,X_{t})dt,\quad \quad X_0=x_0,
\ee
$W$ being  a standard Brownian motion in $\R^d$. We suppose the diffusion coefficients fulfill the following requests:

\bas{assumption1}
There exists a constant $\kappa>0$ such that, $\forall t\in [0,1],\,\forall x \in \R^n$:
\begin{align*}
&\sum_{j=1}^d |\s_j(t,x)| + |b(t,x)| +  \sum_{j=1}^d  \sum_{0\leq |\a| \leq 2}
|\partial_x^{\a}\partial_t\s_j(t,x)|
\leq  \kappa (1+|x|)\\
&\sum_{j=1}^d \sum_{1\leq |\a| \leq 4}
|\partial_x^{\a} \s_j(t,x)|+\sum_{1\leq |\a| \leq 3}
|\partial_x^{\a} b(t,x)|
\leq \kappa
\end{align*}
\eas
Remark that Assumption \ref{assumption1} ensures the strong existence and uniqueness of the solution to \eqref{equation}. We do not assume here ellipticity but a non degeneracy of H\"ormander type. In order to do this, we need to introduce the $n\times d^2$ matrix $A(t,x)$ defined as follows. We set $m=d^2$ and define the function
\be{lip}
l(i,p)=(p-1)d+i\in\{1,\dots,m\},\quad p,i\in\{1,\dots,d\}.
\ee
Notice that $l(i,p)$ is invertible.  For $l=1,\ldots,m$, we set the (column) vector field $A_{l}(t,x)$ in $\R^n$ as follows:
\be{Al}
\begin{split}
A_{l}(t,x) &=[\s_{i},\s_{p}](t,x) \quad \mbox{if}\quad l=l(i,p) \quad \mbox{with}\quad i\neq p, \\
&= \s_{i}(t,x)\quad \mbox{if}\quad l=l(i,p) \quad \mbox{with}\quad i= p
\end{split}
\ee
and we set $A(t,x)$ to be the $n\times m$ matrix whose columns are given by $A_1(t,x),\ldots,A_m(t,x)$:
\be{Alucia}
A(t,x)=[A_1(t,x),\ldots,A_m(t,x)].
\ee
$A(t,x)$ can be interpreted as a directional matrix. We denote by $\lambda (t,x)$ the smallest singular value of $A(t,x)$, i.e.
\be{Not9'}
\lambda (t,x)^2=
\l_*(A(t,x))^2= \inf_{\left\vert \xi \right\vert =1}\sum_{i=1}^{m}\left\langle
A_{i}(t,x),\xi \right\rangle ^{2}.
\ee
In this paper, we assume the following non degeneracy condition. We write it in a ``time dependent way'' because this is  useful in \cite{BCP2}, which represents the second part of the present article. In fact, we use here just $A(0,x_0)$ and $\l(0,x_0)$, whereas in \cite{BCP2} we consider $A(t,x_t)$ and $\l(t,x_t)$, $x_t$ denoting a skeleton path.

\bas{assumption2}
Let $x_0$ denote the starting point of the diffusion $X$ solving \eqref{equation}. We suppose that
\[
\lambda (0,x_0)>0.
\]
\eas

Notice that Assumption \ref{assumption2} is actually equivalent to require that the first order H\"ormander condition holds in the starting point $x_0$, i.e. the vector fields $\sigma _{i}(0,x_0)$, $[\sigma_{j},\sigma _{p}](0,x_0)$, as $i,j,p=1,...,d$, span the whole $\R^{n}$.

\smallskip

We define now the $m\times m$ diagonal scaling matrix $\mathcal{D}_\delta$ as
\[
\begin{split}
(\mathcal{D}_\delta)_{l,l}&=\delta \quad \mbox{if}\quad l=l(i,p) \quad \mbox{with}\quad i\neq p, \\
&= \sqrt{\delta} \quad \mbox{if}\quad l=l(i,p) \quad \mbox{with}\quad i= p
\end{split}
\]
and the scaled directional matrix
\be{Adlucia}
A_\delta(t,x)=A(t,x)\mathcal{D}_\delta.
\ee
Notice that the $l$th column of the matrix $A_\delta(t,x)$ is given by $\sqrt{\delta}\s_{i}(t,x)$ if $l=l(i,p)$ with $ i= p$ and by $\delta[\s_{i},\s_{p}](t,x)=[\sqrt{\delta}\s_{i},\sqrt{\delta}\s_{p}](t,x)$ if $i\neq p$. Therefore, $A_\delta(t,x)$ is the matrix given in \eqref{Alucia} when the original diffusion coefficients $\sigma_j(t,x)$, $j=1,\ldots,d$, are replaced by $\sqrt {\delta}\sigma_j(t,x)$, $j=1,\ldots,d$.

This matrix and the associated norm $|\cdot|_{A_\d(0,x_0)}$ are the tools that allow us to account of the different speeds of propagation of the diffusion: $\sqrt{\d}$ (diffusive scaling) in the direction of $\s$ and $\d$ in the direction of the first order Lie Brackets.
In particuar, straightforward computations easily give that
\be{Norm3}
\frac{1}{\sqrt{\delta}\l^*(A(t,x))}\left\vert y\right\vert \leq
\left\vert y\right\vert _{A_{\delta}(t,x)}\leq \frac{1}{\delta\l_*(A(t,x))}%
\left\vert y\right\vert ,
\ee
We also consider the following assumption, as a stronger version of Assumption \ref{assumption1} (morally we ask for boundedness instead of sublinearity of the coefficients, in the spirit of Kusuoka-Stroock estimates in \cite{KusuokaStroock:87}).
\bas{assumption1bis}
There exists a constant $\kappa>0$ such that for every  $t\in [0,1]$ and $x \in \R^n$ one has
\begin{align*}
\sum_{0\leq |\a| \leq 4} \Big[
\sum_{j=1}^d  |\partial_x^{\a} \s_j(t,x)|
+ |\partial_x^{\a} b(t,x)|
+ |\partial_x^{\a}\partial_t\s_j(t,x)| \Big]
\leq \kappa.
\end{align*}
\eas

The aim of this paper is to prove the following result:
\bt{main-th}
Let Assumption \ref{assumption1} and \ref{assumption2} hold. Let $p_{X_t}$ denote the density of $X_t$, $t>0$, with the starting condition $X_0=x_0$. Then the following holds.
\ben
\item
  There exist positive constants $r,\d_*,C$ such that for every $\d\leq \d_*$ and for every $y$ such that $|y- x_0-b(0,x_0)\d|_{A_\d(0,x_0)}\leq r$,
\[
\frac{1}{C \d^{n-\frac{\dim\langle \s(0,x_0) \rangle}{2} }} \leq
p_{X_\d}(y).
\]
\item For any $p>1$, there exists a positive constant $C$ such that for every $\d\leq 1$ and for every $y\in \R^n$
\[
p_{X_\d}(y)\leq  \frac{1}{\d^{n-\frac{\dim\langle \s(0,x_0) \rangle}{2}}
} \frac{C}{1+|y-x_0|_{A_\d(0,x_0)}^p}.
\]
\item If also Assumption \ref{assumption1bis} holds (boundedness of coefficients) there exists a constant $C$ such that for every $\d\leq 1$ and for every $y\in \R^n$.
\[
p_{X_\d}(y)\leq  \frac{C}{\d^{n-\frac{dim \langle \s(0,x_0) \rangle}{2}}} \exp\big(- \frac{1}{C}|y-x_0|_{A_\d(0,x_0)}\big).
\]
\een
Here $\dim\langle \s(0,x_0) \rangle$ denotes the dimension of the vector space spanned by $\sigma_1(0,x_0),\ldots,$ $\sigma_d(0,x_0)$.
\et

\br{1}
It might appear contradictory that the lower estimate (1) in Theorem \ref{main-th} is centered in $x_0+\d b(x_0)$,
whereas the upper estimates are centered in $x_0$. In fact, this is important  only for the lower bound, the upper bounds (2) and (3) holding true either if we write
$|y-x_0-\d b(x_0)|_{A_\d(0,x_0)}$ or $|y-x_0|_{A_\d(0,x_0)}$ (see next Remark \ref{3}).
\er

\begin{remark}
As already mentioned, the two sided bound \eqref{ksintro} by Kusuoka and Stroock \cite{KusuokaStroock:87} is proved under a strong H\"ormander condition of any order, but the drift coefficient must be generated by the vector fields of the diffusive part, and the diffusion coefficients $b$ and $\sigma_j$, $j=1,\ldots,d$, must not depend on time. Here, on the contrary, we allow for a general drift and time dependence in the coefficients, but we consider only first order Lie Brackets. Moreover, in assumption \ref{assumption1}, we also relax the hypothesis of bounded coefficients. Anyways, the two estimates are strictly related, since  our matrix norm is locally equivalent to the Carath\'eodory distance -- this is proved  \cite{BCP2} (see Section 4 therein).
\er

\begin{remark}
Our main application is developed in \cite{BCP2}, which is the second part of this paper  and concerns tube estimates. To this aim, we are mostly interested in the diagonal estimates, that is, around $x_0+\d b(0,x_0)$. In particular, what we need is the precise exponent
$n-\dim \langle \s(0,x_0) \rangle/2$, which accounts for the time-scale of the heat kernel when $\d$ goes to zero. However, our results are not asymptotic, but hold uniformly for $\d$ small enough. This is crucial for our application to tube estimates, and this is also a main difference with the estimates in \cite{BenArousLeandre:91,BenArousLeandreII:91}.
\er
\begin{remark}
The upper bounds in (2) and (3) of Theorem \ref{main-th} give also the tail estimates, which are exponential if we assume the boundedness of the coefficients, polynomial otherwise.
\end{remark}

The proof of Theorem \ref{main-th} is long, also different according to the lower or upper estimate, and we proceed by organizing two sections where such results will be separately proved. So, the lower estimate will be discussed in Section \ref{lower} and proved in Theorem \ref{lower-th}, whereas Section \ref{upper} and Theorem \ref{upper-th} will be devoted to the upper estimate.

\section{Lower bound}\label{lower}

We study here the lower bound for the density of $X_\delta$.

\subsection{The key-decomposition}\label{lower-key}
We start with the decomposition of the process that will allow us to produce the lower bound in short (but not asymptotic) time.

We first use a development in stochastic Taylor series of order two of the diffusion process $X$ defined through \eqref{equation}. This gives
\begin{equation}
X_{t}=x_0+Z_{t}+b(0,x_0)t+R_{t}  \label{D1bis}
\end{equation}%
where
\begin{equation}\label{Decomp0}
\begin{array}{l}
\displaystyle
Z_t=\sum_{i=1}^{d}a_{i}W_{t}^{i}+\sum_{i,j=1}^{d}a_{i,j}%
\int_{0}^{t}W_{s}^{i}\circ dW_{s}^{j}  \smallskip\\
\mbox{with $a_{i}=\sigma _{i}(0,x_{0})$, $a_{i,j}=\partial _{\sigma _{i}}\sigma _{j}(0,x_{0})$}
\end{array}
\end{equation}%
and
\begin{align}
R_{t}=& \sum_{j,i=1}^{d}\int_{0}^{t}\int_{0}^{s}(\partial _{\sigma
_{i}}\sigma _{j}(u,X_{u})-\partial _{\sigma _{i}}\sigma _{j}(0,x_0))\circ
dW_{u}^{i}\circ dW_{s}^{j}  \label{D2bis} \\
& +\sum_{i=1}^{d}\int_{0}^{t}\int_{0}^{s}\partial _{b}\sigma
_{i}(u,X_{u})du\circ
dW_{s}^{i}+\sum_{i=1}^{d}\int_{0}^{t}\int_{0}^{s}\partial _{u}\sigma
_{j}(u,X_{u})du\circ dW_{s}^{i}  \notag \\
& +\sum_{i=1}^{d}\int_{0}^{t}\int_{0}^{s}\partial _{\sigma
_{i}}b(u,X_{u})\circ dW_{u}^{i}ds+\int_{0}^{t}\int_{0}^{s}\partial
_{b}b(u,X_{u})duds.  \notag
\end{align}%
Since $R_t=\cal{O}(t^{3/2})$, we expect the behavior of $X_t$ and $Z_t$ to be somehow close for small values of $t$. Our first goal is to give a decomposition for $Z_t$ in \eqref{Decomp0}. We start introducing some notation. We fix $\delta >0$ and set
$$
s_{k}(\delta )=\frac{k}{d}\delta,\quad k=1,\ldots,d.
$$
We now consider the following random variables: for $i,k=1,\ldots,d$,
\be{delta1}
\Delta _{k}^{i}(\delta ,W)=W_{s_{k}(\delta )}^{i}-W_{s_{k-1}(\delta
)}^{i},\quad \Delta _{k}^{i,j}(\delta ,W)=\int_{s_{k-1}(\delta
)}^{s_{k}(\delta )}(W_{s}^{i}-W_{s_{k-1}}^{i})\circ dW_{s}^{j}.
\ee%
Notice that $\Delta _{k}^{i,j}(\delta ,W)$ is the Stratonovich integral, but
for $i\neq j$ it coincides with the It\^o integral. When no confusion is
possible we use the short notation $s_{k}=s_{k}(\delta ),\Delta
_{k}^{i}=\Delta _{k}^{i}(\delta ,W),\Delta _{k}^{i,j}=\Delta
_{k}^{i,j}(\delta ,W).$  We also denote the random vector $\Delta(\d,W)$ in $\R^m$
\be{deltal}
\begin{split}
\Delta_{l}(\d,W) &=\Delta_p^{i,p}(\d,W) \quad \mbox{if $l=l(i,p)$ with $i\neq p$}, \\
&= \Delta_p^p(\d,W) \quad \mbox{if $l=l(i,p)$ with $i=p$}.
\end{split}
\ee
(recall $l(i,p)$ in \eqref{lip}). Moreover, with $\sum_{l>p}^{d}=\sum_{p=1}^{d}\sum_{l=p+1}^{d}$, we define
\begin{equation}  \label{Decomp2}
\begin{array}{l}
V(\d,W) =
\displaystyle\sum_{p=1}^d\Big[
\sum_{i\neq p}\Delta^{i}_{p} +
\sum_{i\neq j,i\neq p,j\neq
p}a_{i,j}\Delta _{p}^{i,j}+\sum_{l=p+1}^{d}\sum_{i\neq p}\sum_{j\neq
l} a_{i,j}\Delta _{l}^{j}\Delta _{p}^{i}+\frac{1}{2}\sum_{i\neq
p} a_{i,i}\left\vert \Delta _{p}^{i}\right\vert ^{2}\Big]; \\
\varepsilon _{p}(\delta ,W) = \displaystyle \sum_{l>p}^{d} \sum_{j\neq
l} a_{p,j}\Delta _{l}^{j}+\sum_{p>l}^{d}\sum_{j\neq l} a_{j,p}\Delta
_{l}^{j}+\sum_{j\neq p} a_{p,j}\Delta _{p}^{j},\quad p=1,...,d; \\
\eta _{p}(\delta ,W) = \displaystyle\frac{1}{2}a_{p,p}\left\vert \Delta
_{p}^{p}\right\vert ^{2}+ \sum_{l>p}^{d} a_{p,l}\Delta _{l}^{l}\Delta
_{p}^{p}+\Delta _{p}^{p}\varepsilon _{p}(\delta,W),\quad p=1,...,d.%
\end{array}%
\end{equation}
We have the following decomposition:
\bl{decZ-l}
Let $\Delta(\delta,W)$ and $A(0,x_0)$ be given in \eqref{deltal} and \eqref{Alucia} respectively. One has
\be{decZ}
Z_\delta=
V(\d,W)+A(0,x_0) \Delta(\d,W)+\eta(\delta ,W),
\ee
where $V(\delta,W)$ is given in \eqref{Decomp2} and $\eta (\delta ,W)=\sum_{p=1}^{d}\eta _{p}(\delta ,W)$, $\eta_p(\delta ,W)$ being given in \eqref{Decomp2}.

\el

The proof of Lemma \ref{decZ-l} is quite long, so it is postponed to Appendix \ref{app-dec}.

\br{111}
The reason of this decomposition is the following. We split the
time interval $(0,\delta )$ in $d$ sub intervals of length $\delta /d.$ We also split the Brownian motion in corresponding increments: $%
(W_{s}^{p}-W_{s_{k-1}}^{p})_{s_{k-1}\leq s\leq s_{k}},p=1,...,d.$ Let us fix
$p.$ For $s\in (s_{p-1},s_{p})$ we have the processes $%
(W_{s}^{i}-W_{s_{p-1}}^{i})_{s_{p-1}\leq s\leq s_{p}},i=1,...,d.$ Our idea
is to settle a calculus which is based on $W^{p}$ and to take conditional
expectation with respect to $W^{i},i\neq p.$ So $%
(W_{s}^{i}-W_{s_{p-1}}^{i})_{s_{p-1}\leq s\leq s_{p}},i\neq p$ will appear
as parameters (or controls) which we may choose in an appropriate way. The random variables on which the calculus is based are $\Delta
_{p}^{p}=W_{s_{p}}^{p}-W_{s_{p-1}}^{p}$ and $\Delta
_{p}^{i,p}=\int_{s_{p-1}}^{s_{p}}(W_{s}^{i}-W_{s_{p-1}}^{i})dW_{s}^{p},j\neq
p.$ These are the r.v. that we have emphasized in the
decomposition of $Z_\delta.$ Notice that, conditionally to the controls $%
(W_{s}^{i}-W_{s_{p-1}}^{i})_{s_{p-1}\leq s\leq s_{p}},i\neq p,$ this is a
centered Gaussian vector and, under appropriate hypothesis on the controls
this Gaussian vector is non degenerate (we treat in section \ref{app-supp} the
problem of the choice of the controls).
In order to handle the term $\Delta_{p}^{p,i}=\int_{s_{p-1}}^{s_{p}}(W_{s}^{p}-W_{s_{p-1}}^{p})dW_{s}^{i}.$
we use the identity $\Delta _{p}^{p,i}=\Delta _{p}^{i}\Delta _{p}^{p}-\Delta
_{p}^{i,p}$.

\er

We now emphasize the scaling in $\delta$ in the random vector $\Delta(\delta,W)$.
We define
$B_{t}=\delta ^{-1/2}W_{t\delta }$
and denote
\be{New3}
\begin{split}
\Th_l &=\frac{1}{\d }\Delta
_{p}^{i,p}=\int_{\frac{p-1}{d}}^{\frac{p}{d}}(B_{s}^{i}-B_{\frac{p-1}{d}}^{i})dB_{s}^{p}\quad \mbox{if $l=l(i,p)$ with $i\neq p$}, \\
&= \frac{1}{\sqrt{\d }}\Delta_{p}^{p}=B_{\frac{p}{d}}^{p}-B_{\frac{p-1}{d}}^{p} \quad \mbox{if $l=l(i,p)$ with $i=p$},
\end{split}
\ee
$l(i,p)$ being given in \eqref{lip}. For $p=1,...,d$ we denote with $\Theta_{(p)}$ the $p$th block of $\Theta$ with length $d$, that is
$$
\Theta _{(p)}=(\Theta _{(p-1)d+1},...,\Theta
_{pd}),
$$
so that $\Theta=(\Theta _{(1)},\ldots,\Theta _{(d)})$.
We will also denote
\be{New3bis}
l(p)=l(p,p)=(p-1)d+p\quad\mbox{and}\quad\Th_{l(p)}=\frac{1}{\sqrt{\d }
}\Delta_{p}^{p}.
\ee

Consider now the $\s$ field
\begin{equation}
\mathcal{G}:=\sigma (W_{s}^{j}-W_{s_{p-1}(\delta )}^{j},s_{p-1}(\delta )\leq
s\leq s_{p}(\delta ),p=1,...d,j\neq p).  \label{New4}
\end{equation}%
Then conditionally to $\mathcal{G}$ the random variables $%
\Theta _{(p)},p=1,...,d$ are independent centered Gaussian $d$
dimensional vectors and the covariance matrix $Q_{p}$ of $\Theta _{(p)}$
is given by
\begin{equation}\label{Qp}
\begin{array}{rl}
Q_{p}^{p,j}&
\displaystyle
=Q_{p}^{j,p}=\int_{\frac{p-1}{d}}^{\frac{p}{d}}\left(B_{s}^{j}-B_{\frac{p-1}{d}}^{j}\right)ds,\quad j\neq p,\smallskip\\
Q_{p}^{i,j}&
\displaystyle
=\int_{\frac{p-1}{d}}^{\frac{p}{d}}\left(B_{s}^{j}-B_{\frac{p-1}{d}}^{j}\right)\left(B_{s}^{i}-B_{\frac{p-1}{d}}^{i}\right)ds,\quad j\neq
p,i\neq p,  \smallskip\\
Q_{p}^{p,p}& =\frac{1}{d}.
\end{array}
\end{equation}
It is easy to see that $\det Q_{p}\neq 0$ almost surely.
It follows that conditionally to $\mathcal{G}$ the random variable $\Theta
=(\Theta _{(1)},...,\Theta _{(d)})$ is a centered $m=d^{2}$ dimensional
Gaussian vector. Its covariance matrix $Q$ is a block-diagonal matrix built with $Q_{p},\,p=1,\dots,d$:
\begin{equation}\label{Q}
Q=
\left(
\begin{array}{ccc}
Q_1 &       & \\
    &\ddots & \\
    &       & Q_d
\end{array}
\right)
\end{equation}
In particular $\det Q=\prod_{p=1}^{d}\det Q_{p}\neq 0$
almost surely, and
$\lambda _{\ast }(Q)=\min_{p=1,...,d}\lambda _{\ast }(Q_p)$. We also have $\lambda ^{\ast }(Q)=\max_{p=1,...,d}\l^{\ast }(Q_p)$. We will need to work on subsets where we have a quantitative control of this quantities, so we will come back soon on these covariance matrices. But let us show now how one can rewrite decomposition \eqref{decZ} in terms of the random vector $\Theta$. As a consequence, the scaled matrix $A_\delta=A_\delta(0,x_0)$ in \eqref{Adlucia} will appear.

We denote by $A_{\delta }^{i}\in \R^{m},i=1,...,n$ the rows of the matrix $%
A_{\delta }$. We also denote $S=\langle A_{\delta }^{1},...,A_{\delta
}^{n}\rangle \subset \R^{m}$ and $S^{\bot }$ its orthogonal. Under Assumption \ref{assumption2}  the columns of $A_{\delta }$ span $\R^{n}$
so the rows $A_{\delta }^{1},...,A_{\delta
}^{n}$ are linearly independent and $S^{\bot }$ has
dimension $m-n$. We take $\Gamma _{\delta }^{i},i=n+1,...,m$ to be an
orthonormal basis in $S^{\bot}$ and we denote $\Gamma _{\delta
}^{i}=A_{\delta }^{i}(0,x_0)$ for $i=1,...,n$.
We also denote $\underline{\G}_{\d}$ the $(m-n) \times m $ matrix with rows  $\G_\d^i,i=n+1, \dots, m$.
Finally we denote by $\Gamma
_{\delta }$ the $m\times m$ dimensional matrix with rows $\Gamma _{\delta
}^{i},i=1,...,m$.
Notice that
\begin{equation}
\Gamma _{\delta }\Gamma _{\delta }^T=\left(
\begin{tabular}{ll}
$A_{\delta }A_{\delta }^T(0,x_0)$ & $0$ \\
$0$ & $\mathrm{Id}_{m-n}$%
\end{tabular}%
\right)  \label{New7}
\end{equation}%
where $\mathrm{Id}_{m-n}$ is the identity matrix in $\R^{m-n}.$ It follows that for a
point $y=(y_{(1)},y_{(2)})\in \R^{m}$ with $y_{(1)}\in \R^{n},y_{(2)}\in
\R^{m-n}$ we have%
\be{GammaA}
\left\vert y\right\vert _{\Gamma _{\delta }}^{2}=\left\vert
y_{(1)}\right\vert _{A_{\delta }(0,x_0)}^{2}+\left\vert y_{(2)}\right\vert ^{2}
\ee
where we recall that $\left\vert y\right\vert _{\Gamma _{\delta
}}^{2}=\left\langle (\Gamma _{\delta }\Gamma _{\delta }^T)^{-1}y,y\right\rangle$.
For $a\in \R^{m}$ we define the immersion
\begin{equation}
J_{a}:\R^{n}\rightarrow \R^{m},\quad
(J_{a}(z))_{i}=z_{i},i=1,...,n\quad\mbox{and}\quad (J_{a}(z))_{i}=\left\langle \Gamma _{\delta
}^{i},a\right\rangle ,i=n+1,...,m.  \label{New9}
\end{equation}%
In particular $J_{0}(z)=(z,0,...,0)$ and
\begin{equation}
\left\vert J_{0}z\right\vert _{\Gamma _{\delta }}=\left\vert z\right\vert
_{A_{\delta }(0,x_0)}.  \label{New9a}
\end{equation}%
Finally we denote%
\begin{equation}\label{New10}
\begin{array}{rcl}
V_{\omega } &=&V(\delta,W) \smallskip\\
\eta _{\omega }(\Theta ) &=&
\displaystyle
\sum_{p=1}^{d}\left(\frac{a_{p,p}}{2}\delta \Theta
_{l(p)}^{2}+\delta ^{1/2}\Theta _{l(p)}\varepsilon
_{p}(\delta,W)+\sum_{q>p}^{d}a_{p,q}\delta \Theta _{l(q)}\Theta _{l(p)}\right)
\end{array}
\end{equation}%
where $V(\delta,W)$ and $\varepsilon_{p}(\delta ,W)$ are
defined in (\ref{Decomp2}) and $\Theta _{l(p)}$ is given in \eqref{New3bis}. We notice that $\eta _{\omega }(\Theta )=\sum_{p=1}^{d}\eta _{p}(\delta
,W)$, $\eta_{p}(\delta ,W)$ being defined in (\ref{Decomp2}).
We also remark that both $V(\delta,W)$ and $\varepsilon_{p}(\delta ,W)$ are $\mathcal{G}$-measurable, so
\eqref{New10} stresses a dependence on $\omega$ which is $\mathcal{G}$-measurable and a dependence on the random vector $\Theta$ whose conditional law w.r.t. $\mathcal{G}$ is Gaussian.

Now the decomposition \eqref{decZ} may be written as
\[
Z_\delta=V_{\omega }+A_{\delta }(0,x_0)\Theta +\eta _{\omega }(\Theta ).
\]
We embed this relation in $\R^{m}$ and obtain%
\[
J_{\Theta }(Z_\delta)=J_{0}(V_{\omega })+\Gamma _{\delta }\Theta
+J_{0}(\eta _{\omega }(\Theta )).
\]
We now multiply with $\Gamma _{\delta }^{-1}$: setting
\be{deftilde}
\widetilde{Z}=\Gamma _{\delta }^{-1}J_{\Theta }(Z_\delta),\quad \widetilde{V}_{\omega }=\Gamma _{\delta }^{-1}J_{0}(V_{\omega }),\quad
\widetilde{\eta }_{\omega }(\Theta )=\Gamma _{\delta
}^{-1}J_{0}(\eta _{\omega }(\Theta ))
\ee
and
\be{defG}
G=\Theta +\widetilde{\eta }_{\omega}(\Theta),
\ee
we get
\be{New11}
\widetilde{Z}=\widetilde{V}_{\omega }+G.
\ee
Notice that, conditionally to $\mathcal{G}$, $\widetilde Z$ is a translation of the random variable $G=\Theta+\widetilde{\eta }_{\omega}(\Theta)$ which is a perturbation of a centred Gaussian random variable. Thanks to this fact, we can to apply the results in Appendix \ref{sectioninversefunction}: we use a local inversion argument in order to give bounds for the conditional density of $\widetilde Z$, which will be used in order to get bounds for the non conditional density. As a consequence, we will get density estimates for $Z_\delta$.

\subsection{Localized density for the principal term $\widetilde Z$}
We study here the density of $\widetilde Z$ in \eqref{New11}, ``around'' (that is, localized on) a suitable set of Brownian trajectories, where we have a quantitative control on the ``non-degeneracy'' (conditionally to $\mathcal{G}$) of the main Gaussian $\Th$.

We denote
\begin{equation}\label{qp}
q_{p}(B)=\sum_{j\neq p}\left\vert B_{\frac{p}{d}}^{j}-B_{\frac{p-1}{d}}^{j}\right\vert
+\sum_{j\neq p,i\neq p,i\neq j}\left\vert
\int_{\frac{p-1}{d}}^{\frac{p}{d}}(B_{s}^{j}-B_{\frac{i-1}{d}}^{j})dB_{s}^{i}\right\vert. 
\end{equation}
For fixed $\ve,\rho>0$, we define
\begin{equation}\label{def-Lambda}
\begin{array}{l}
\Lambda _{\r,\varepsilon ,p}=\Big\{\det Q_{p}\geq \varepsilon^\r,\sup_{\frac{p-1}{d}\leq t\leq \frac{p}{d}}\sum_{j\neq p}\vert
B_{t}^{j}-B_{\frac{p-1}{d}}^{j}\vert \leq \varepsilon^{-\r},q_{p}(B)\leq
\varepsilon \Big\},\ p=1,\ldots,d\smallskip\\
\Lambda _{\r,\varepsilon}=\cap_{p=1}^d\Lambda _{\r,\varepsilon ,p}.
\end{array}
\end{equation}%
Notice that $\Lambda _{\r,\varepsilon ,p}\in\mathcal{G}$ for every $p=1,\ldots,d$.
By using some results in Appendix \ref{app-supp}, we get the following.
\begin{lemma}\label{lemma-Lambda}
Let $\Lambda _{\r,\varepsilon }$ be as in \eqref{def-Lambda}.
There exist $c$ and $\varepsilon _{\ast }$ such that for every $\varepsilon\leq \varepsilon_\ast$ one has
\begin{equation}
\PR\big(\Lambda _{\r,\varepsilon }\big)\geq c\times \varepsilon ^{\frac{1}{2%
}m(d+1)}.  \label{N9}
\end{equation}%
\end{lemma}

\bpr
We apply here Proposition \ref{SUPORT3}. Let $p\in\{1,\ldots d\}$ be fixed and consider the Brownian motion $\widehat B_t=\sqrt d(B_{\frac{p-1+t}d}-B_{\frac{p-1}d})$. Let $Q(\widehat B)$ be the matrix in \eqref{Q-app}. Up to a permutation of the components of $\widehat B$, we easily get $Q^{p,p}(\widehat B)=d\times Q_p^{p,p}$, $Q^{p,j}(\widehat B)=d^{3/2}\times Q_p^{p,j}$ for $j\neq p$, $Q^{i,j}(\widehat B)=d^2\times Q_p^{i,j}$ for $i\neq p$ and $j\neq p$. Therefore,
$$
\det Q_p=d^{2d-1}\det Q(\widehat B)\geq \det Q(\widehat B).
$$
Let now $q(\widehat B)$ be the quantity defined in \eqref{Sup4}. With $q_p(B)$ as in \eqref{qp}, it easily follows that
$$
q_p(B)\leq  q(\widehat B).
$$
Moreover, $\sup_{t\leq 1}|\widehat B_t|=\sqrt d\, \sup_{\frac{p-1}d\leq s\leq \frac pd}|B_s-B_{\frac {p-1}d}|\geq \sup_{\frac{p-1}d\leq s\leq \frac pd}|B_s-B_{\frac {p-1}d}|$.
As a consequence, with $\Upsilon_{\rho ,\varepsilon }$ the set defined in \eqref{Sup3}, we get
$$
\Upsilon_{\rho ,\varepsilon }(\widehat B)
\subset \Lambda_{\rho, \varepsilon, p}
$$
and by using  \eqref{Sup2},  we may find some constants $c$ and $%
\varepsilon _{\ast }$ such that
$\PR(\Lambda _{\r,\varepsilon ,p})\geq c\varepsilon ^{\frac{1}{2}%
d(d+1)}$, for $\varepsilon \leq \varepsilon _{\ast }$.
This holds for every $p$. Since $\Lambda _{\r,\varepsilon }=\cap _{p=1}^{d}\Lambda _{\r,\varepsilon ,p}$, by using the independence property we get \eqref{N9}.
\epr

Let $Q$ be the matrix in \eqref{Q}. On the set $\Lambda _{\r,\varepsilon }\in \cal{G}$ we have $\det Q=\prod_{p=1}^{d}\det Q_{p}\geq \varepsilon ^{d\r}$. Remark that
\be{complambdanorm}
\frac{\lambda^{\ast }(Q)}{\sqrt{m}}
\leq |Q|_l := \left(\frac{1}{m}\sum_{1\leq i,j\leq m} Q_{i,j}^2 \right)^{1/2}
\leq \lambda^{\ast }(Q).
\ee
For $a>0$ we introduce the following function,
\[
\psi_a(x)=1_{|x|\leq a}+\exp\left( 1-\frac{a^2}{a^2-(x-a)^2} \right)1_{a<|x|<2a},
\]
which is a mollified version of $1_{[0,a]}$.
We can now define our \emph{localization variables}.
\begin{equation}\label{Utilde}
\widetilde{U}_\varepsilon=(\psi_{a_1}(1/\det Q))
\psi_{a_2}(|Q|_l)
\psi_{a_3}(q(B)),\quad\mbox{with}\quad a_1=\ve^{-d\r},\ a_2=\ve^{-2\r}, \ a_3=d\ve
\end{equation}
in which we have set
$$
q(B)=\sum_{p=1}^dq_p(B).
$$
Remark that
$\widetilde{U}_\varepsilon$ is measurable w.r.t. $\cal{G}$.
The following inclusions hold: for every $\varepsilon$ small enough,
\be{N10}
\Lambda _{\r,\varepsilon }
\subset
\Big\{\det Q \geq \varepsilon^{d\r},|Q|_l\leq \ve^{-2\r},q(B)\leq
d \varepsilon \Big\}= \{\widetilde{U}_\varepsilon= 1\}\subset \{\widetilde{U}_\varepsilon\neq 0\}.
\ee
We can consider $\widetilde{U}_\varepsilon$ as a smooth version of the indicator function of $\Lambda_{\r,\ve}$.
We also define, for fixed $r>0$,
\begin{equation}\label{Ubar}
\bar{U}_r=\prod_{i=1}^n \psi_r(\Th_i).
\end{equation}

In order to state a lower estimate for the (localized) density of $\widetilde Z$ in \eqref{New11}, we define the following set of constants:
\be{CC}
\mathcal{C}=\Big\{C>0\,:\,C=\exp\Big({c \Big(\frac{\kappa}{\lambda(0,x_0)}\Big)^{q}}\Big),\ \exists\ c,q>0\Big\}
\ee
and we set
$$
1/\mathcal{C}=\{C>0\,:\,1/C\in\mathcal{C}\}.
$$

\bl{lemmaZ}
Suppose Assumption \ref{assumption1} and \ref{assumption2} both hold. Let $U_{\varepsilon,r}=\widetilde U_\varepsilon\bar U_r$, $\widetilde U_\varepsilon$ and $\bar U_r$ being defined in \eqref{Utilde} and \eqref{Ubar} respectively, with $\rho=\frac 1{8m}$. Set $d\PR_{U_{\varepsilon,r}}=U_{\varepsilon,r}d\PR$ and let $p_{\widetilde{Z},U_{\varepsilon,r}}$ denote  the density of $\widetilde Z$ in \eqref{New11} when we endow $\Omega$ with the measure $\PR_{U_{\varepsilon,r}}$.
Then there exist $C\in \cal{C}$, $\ve,r\in 1/\cal{C}$ such that for $|z|\leq r/2$,
\be{estZ}
p_{\widetilde{Z},U_{\varepsilon,r}}(z)\geq \frac{1}{C}.
\ee
\el

\bpr
STEP 1: lower bound for the localized conditional density given $\mathcal{G}$.

\smallskip

Let $p_{\widetilde{Z},\bar{U}_r|\cal{G}}$ denote the localized density w.r.t. the localization $\bar U_r$ of $\widetilde{Z}$ conditioned to $\cal{G}$, i.e.
\be{defcond}
\E[f(\widetilde{Z})\bar{U}_r|\cal{G}] = \int f(z) p_{\widetilde{Z},\bar{U}_r|\cal{G}}(z)dz,
\ee
for $f$ positive, measurable, with support included in $B(0,r/2)$.
We start proving that there exist $C\in \cal{C}$, $\ve,r \in 1/\cal{C}$ such that, on $\widetilde{U}_\varepsilon\neq 0$, for $|z|\leq r/2$
\[
p_{\widetilde{Z},\bar{U}_r|\cal{G}}(z)\geq \frac{1}{C}.
\]
We recall \eqref{New11}: $\widetilde Z=\widetilde V_\omega+\Theta+\widetilde\eta_\omega(\Theta)$, where $\omega\mapsto\widetilde V_\omega$ and $\omega\mapsto\widetilde\eta_\omega(\cdot)$ are both $\mathcal{G}$-measurable and the conditional law of $\Theta$ given $\mathcal{G}$ is Gaussian. This allows us to use the results in Appendix \ref{sectioninversefunction}. In particular, we are interested in working on the set $\{\widetilde{U}_\varepsilon\neq 0\}\in\mathcal{G}$, so one has to keep in mind that $\omega\in\{\widetilde{U}_\varepsilon\neq 0\}$.

\smallskip

On $\widetilde{U}_\varepsilon\neq 0$, by \eqref{Utilde} and \eqref{complambdanorm} one has $\l^*(Q) \leq 2 \sqrt{m} \ve^{-2\r}$, and
\begin{equation*}
\frac{\varepsilon ^{d\rho } }{2} \leq \det Q\leq \lambda _{\ast }(Q)\lambda ^{\ast}(Q)^{m-1} \leq \lambda _{\ast }(Q) (2 \sqrt{m})^{m-1}\ve ^{-2\rho (m-1)},
\end{equation*}%
and this gives $\lambda _{\ast }(Q)\geq \frac{ \varepsilon ^{3 m \rho } }{(2\sqrt{m})^m} $. So, fixing $\rho=1/(8m)$, for $\ve \leq \ve^*$,
\be{rlh}
\frac{1}{16 m^2}\frac{\l_*(Q)}{\l^*(Q)} \geq
C_m \varepsilon ^{3 m \rho+2\rho } \geq
\ve.
\ee
To apply \eqref{invest} to $G=\Theta +\widetilde{\eta }_{\omega}(\Theta)$ we need to check the hypothesis of Lemma \ref{invfun2}. We are going to use the notation of Appendix \ref{sectioninversefunction}, in particular for $c_*(\widetilde{\eta}_\o,r)$ in \eqref{c*} and $c_i(\widetilde{\eta}_\o)$, $i=2,3$, in \eqref{ci}. Recall that $\widetilde\eta_\omega$ is defined in \eqref{deftilde} through $\eta_\omega$ given in \eqref{New10}. Since the third order derivatives of $\eta_\o$ are null,  we have $c_{3}(\widetilde{\eta }_{\o})=0$. Also, for $i=l(p)$ and $j=l(q)$ we have $\partial _{i,j}\eta _{\o }(\Th )=\d a_{ij}$, otherwise we get $\partial _{i,j}\eta _{\o }(\Th )=0$. So $\left|\partial _{i,j}\eta _{\o}(\Th)\right| \leq \d \sum_{i,j} |a_{i,j}|$. Using \eqref{Norm3} we obtain
\begin{equation*}
\left| \partial _{i,j}\widetilde{\eta }_{\o}(\Th)\right|
=\left| J_{0}(\partial _{i,j}\eta _{\o
}(\Th))\right|_{\G_{\d}}
=\left| \partial_{i,j}\eta_{\o}(\Th)\right|_{A_{\d}}
\leq \frac{\sum_{i,j} |a_{i,j}|}{\lambda(0,x_0)}\leq C\in \cal{C}.
\end{equation*}%
So, with $h_{\eta_\omega}$ as in \eqref{defheta}, we get
\be{c2}
h_{\eta_\omega}=\frac{1}{16 m^2 (c_2(\widetilde{\eta }_{\o})+\sqrt{c_3(\widetilde{\eta }_{\o})})}\geq \frac{1}{C_1}, \quad \exists C_1\in \cal{C}
\ee
We compute now the first order derivatives. For $j\notin \{l(p):p=1,...,d\}$
we have $\partial _{j}\eta _{\omega }=0$ and for $j=l(p)$ we have
\[
\partial_j\eta _{\omega }(\Theta )=\d \sum_{q=p}^{d}a_{p\wedge
q,p\vee q}\Theta _{l(q)}+\sqrt{\delta }\ve_{j}(\delta,W).
\]
So, as above, we obtain $\left| \partial _{j}\widetilde{\eta }_{\o}(\Th)\right|
\leq C (|\Th|+|\ve_j(\delta,W)|/\sqrt{\d})$. Remark now that on $\{\bar{U}_r\neq 0\}$ we have $|\Th|\leq C r$, and on $\{\widetilde{U}_\varepsilon\neq 0\}$ we have $q(B)\leq 2d\ve$, so
\be{ept}
\sum_{j=1}^d |\ve_j(\delta,W)|\leq C\sqrt{\d} q(B) \leq C\sqrt{\d} \ve.
\ee
Therefore
\be{cstar}
c_*(\widetilde{\eta }_{\o},16r)\leq C_2 (r +\ve), \quad \exists C_2\in \cal{C}.
\ee
We also consider the following estimate of $|\widetilde{V}_{\omega }| =|V_{\omega }|_{A_{\delta }}$. First, we rewrite $V_\omega$ as follows:
$$
\begin{array}{c}
\displaystyle
V_{\omega }
=\sum_{p}a_p\mu_p(\delta,W)+\sum_p\psi_p(\delta, W),\quad\mbox{with}\\
\displaystyle
\mu _{p}(\delta ,W)
 =  \displaystyle\sum_{i\neq p}\Delta _{i}^{p}
\mbox{ and }
\psi _{p}(\delta ,W) =
\!\!\!\!\!\!\!\sum_{i\neq j,i\neq p,j\neq
p}\!\!\!\!\!\!a_{i,j}\Delta _{p}^{i,j}+\!\!\sum_{l=p+1}^{d}\sum_{i\neq p}\sum_{j\neq
l} a_{i,j}\Delta _{l}^{j}\Delta _{p}^{i}+\frac{1}{2}\sum_{i\neq
p} a_{i,i}\left\vert \Delta _{p}^{i}\right\vert ^{2}
\end{array}
$$
Using again \eqref{Norm3} we have
\[
\Big|\sum_{p=1}^{d}a_{p}\mu _{p}(\delta ,W)\Big| _{A_{\delta }}=%
\frac{1}{\sqrt{\delta }}\Big| A_{\delta }J_{0}\Big(\sum_{p=1}^{d}\mu _{p}(\delta
,W)\Big)\Big|_{A_{\delta }}\leq \sum_{p=1}^{d}\frac{1}{\sqrt{\delta }}
|\mu_{p}(\delta ,W))|\leq C q(B)
\]
and
\[
\left\vert \psi (\delta ,W)\right\vert _{A_{\delta }}\leq \frac{\left\vert
\psi (\delta ,W)\right\vert}{\delta \sqrt{\lambda (0,x_0)}}
\leq C q(B).
\]
Since $\omega\in\{\widetilde{U}_\varepsilon\neq 0\}$ we get
\be{V}
\begin{split}
|\widetilde{V}_{\omega }| \leq C q(B) \leq C_3 \ve,\quad \exists C_3\in \cal{C}.
\end{split}
\ee
We consider \eqref{V}, and fix $\frac{r}{\ve}=2 C_3 \in \cal{C}$, so $|\widetilde{V}_{\o}|\leq r/2$. Then we consider \eqref{cstar} and we obtain
\[
c_*(\widetilde{\eta }_{\o},4r)\leq C_2 (2 C_3 +1)\ve \leq \ve^{1/2}, \quad \mbox{for }\, \ve\leq  \frac{1}{(4 C_2 C_3)^2} \in \frac{1}{\cal{C}}.
\]
Moreover, looking at \eqref{c2}
\[
r=2 C_3\ve \leq \frac{1}{C_1}\quad \mbox{for }\, \ve \leq \frac{1}{2 C_1 C_3}\in \frac{1}{\cal{C}}.
\]
So, with
\[
\ve=\ve^*\wedge \frac{1}{(4 C_2 C_3)^2} \wedge \frac{1}{2 C_1 C_3}\in \frac{1}{\cal{C}},
\]
and $r=2 C_3 \ve$ we have
\[
|\widetilde{V}_{\o}|\leq \frac{r}{2},
\qquad  c_*(\widetilde{\eta}_\o,4 r)\leq \ve^{1/2},
\qquad r\leq \frac{1}{C_1}.
\]
Now, by using also \eqref{rlh} and \eqref{c2}, it follows that \eqref{hpimpl} holds, and we can apply Lemma \ref{invfun2}. We obtain
\[
\begin{split}
\frac{1}{K \det Q^{1/2}
}\exp\left(-\frac{K}{\l_*(Q)}|z|^2 \right) \leq
p_{G,\bar{U}_r|\cal{G}}(z)
\end{split}
\]
for $|z|\leq r$, where $K$ does not depend on $\s,b$. Remark that, using $\l_*(Q)\geq \frac{\ve^{3m\r}}{(2\sqrt{m})^m}$, $\r=\frac{1}{8m}$, $r/\ve=2 C_1$ and $\ve\leq 1/(4 C_2 C_1)^2$,
\be{noexp}
\frac{|z|^2}{\l_*(Q)}\leq \frac{(2\sqrt{m})^m r^2}{\ve^{3m\r}}\leq
(2\sqrt{m})^m \frac{ r^2}{\ve}
\leq
(2\sqrt{m})^m \frac{ r^2}{\ve^2}\ve
\leq
(2\sqrt{m})^m (2C_1)^2 \ve \leq \bar{K}
\ee
where $\bar{K}$ does not depend on $\s$, $b$. Therefore $ p_{G,\bar{U}_r|\cal{G}}(z)\geq \frac{1}{C} $, for $|z|\leq r$, for some $C\in\cal{C}$, on $\widetilde{U}_\varepsilon\neq 0$. Now, by recalling that $|\widetilde{V}_\o|\leq r/2$ and \eqref{New11}, we have
\be{Zc}
p_{\widetilde{Z},\bar{U}_r|\cal{G}}(z)\geq \frac{1}{C},\quad\mbox{for $|z|\leq r/2$ on the set $\{\widetilde U_\varepsilon\neq 0\}$}.
\ee
STEP 2: we get rid of the conditioning on $\cal{G}$, to have non-conditional bound for $p_{\widetilde{Z},U_{\varepsilon,r}}$.

Since $\widetilde{U}_\varepsilon$ is $\cal{G}$ measurable, for every non-negative and measurable function $f$  with support included in $B(0,r/2)$ we have
\[
\E(f(\widetilde{Z})U_{\varepsilon,r})=\E\big(\widetilde{U}_\varepsilon
\E(f(\widetilde{Z})\bar{U}_r|\cal{G})\big).
\]
By \eqref{defcond} and \eqref{Zc}, we obtain
\[
\E(f(\widetilde{Z})U_{\varepsilon,r})\geq \frac{1}{C}\E(\widetilde{U}_\varepsilon) \int f(z) dz
\]
Since $\Lambda_{\r, \ve} \subset \{\widetilde{U}_\varepsilon= 1\}$, $\E(\widetilde{U}_\varepsilon)\geq \PR(\Lambda_{\r, \ve})$, so by using \eqref{N9} and $\ve\in 1/\cal{C}$ we finally get that $\E(\widetilde{U}_\varepsilon)\geq \frac{1}{C}$, so \eqref{estZ} is proved.
\epr

\subsection{Lower bound for the density of $X_\delta$}

We study here a lower bound for the density of $X_\delta$, $X$ being the solution to \eqref{equation}. Recall decomposition \eqref{D1bis}:
$$
X_\delta-x_0-b(0,x_0)\delta=Z_\delta+R_\delta.
$$
Our aim is to ``transfer'' the lower bound for $\widetilde Z=\Gamma^{-1}_\d J_\Theta(Z_\delta)$ already studied in Lemma \ref{lemmaZ} to a lower bound for $X_\delta$. In order to set up this program, we use results on the distance between probability densities which have been  developed in \cite{BallyCaramellino:12}. In particular, we are going to use now Malliavin calculus. Appendix \ref{app-mall} is devoted to a recall of all the results and notation the present section is based on. In particular, we denote with $D$ the Malliavin derivative with respect to $W$, the Brownian motion driving the original equation \eqref{equation}.

But first of all, we need some properties of the matrix $\Gamma_\delta$, which can be resumed as follows. We set $SO(d)$ the set of the $d\times d$ orthogonal matrices and we denote with $\mathrm{Id}_d$ the $d\times d$ identity matrix.

\begin{lemma}\label{Gamma-prop}
Set for simplicity $A_\delta=A_\delta(0,x_0)$ and let $\Gamma_\delta$ be as in \eqref{New7}. There exist $\mathcal{U}\in SO(n)$, $\underline{\mathcal{U}}\in SO(m-n)$ and $\mathcal{V}\in SO(m)$ such that
\[
\Gamma_\d = \left(
\begin{array}{cc}
\mathcal{U} & 0 \\
0^T & \underline{\mathcal{U}}
\end{array}
\right) \left(
\begin{array}{cc}
\bar{\Sigma} & 0 \\
0^T & \mathrm{Id}_{m-n}
\end{array}
\right) \mathcal{V}^T
\]
where $0$ denotes a null $n\times (m-n)$ matrix and $\bar{\Sigma}=\bar{\Sigma}=\mathrm{Diag}(\lambda_1(A_\delta),\ldots,\lambda_n(A_\delta))$, $\lambda_i(A_\delta)$, $i=1,\ldots,n$, being the singular values of $A_\delta$ (which are strictly positive because $A_\delta$ has full rank).
\end{lemma}

\begin{proof}
We recall that
\[
\Gamma_\d=\left( \begin{array}{c}A_\d\\
\underline{\Gamma}_\d
\end{array}\right),
\]
where $\underline{\Gamma}_\d$ is a $(m-n)\times n$ matrix whose rows are vectors of $\R^m$ which are orthonormal  and orthogonal with the rows of $A_\d$. We take a singular value decomposition for $A_\delta$ and for $\underline{\Gamma}_\delta$. So, there exist $\mathcal{U}\in SO(n)$ and $\bar{\mathcal{V}}\in SO(m)$ such that
$$
A_\d= \mathcal{U} \big(\bar{\Sigma}\ 0\big) \bar{\mathcal{V}}^T,
$$
$0$ denoting the $n\times (m-n)$ null matrix. Similarly, there exist
$\underline{\mathcal{U}}\in SO(m-n)$ and $\underline{\mathcal{V}}\in SO(m)$ such that
$$
\underline{\Gamma}_\d= \underline{\mathcal{U}} \big(0^T\ \mathrm{Id}_{m-n}\big) \underline{\mathcal{V}}^T,
$$
the diagonal matrix being $\mathrm{Id}_{m-n}$ because the rows of $\underline{\Gamma}_\d$ are orthonormal. Therefore, we get
\[
\Gamma_\d = \left(
\begin{array}{cc}
\mathcal{U} & 0 \\
0^T & \underline{\mathcal{U}}
\end{array}
\right) \left(
\begin{array}{cc}
\bar{\Sigma} & 0 \\
0^T & \mathrm{Id}_{m-n}
\end{array}
\right) \mathcal{V}^T
\]
where $\mathcal{V}$ is a $m\times m$ matrix whose first $n$ columns are given by the first $n$ columns of $\overline{\mathcal{V}}$ and the remaining $m-n$ columns are given by the last $m-n$ columns of $\underline{\mathcal{V}}$. Moreover, since each row of $A_\delta$ is orthogonal to any row of $\underline{\Gamma_\delta}$, it immediately follows that all columns of $\mathcal{V}$ are orthogonal. This proves that $\mathcal{V}\in SO(m)$, and the statement follows.
\end{proof}

Then we have

\bl{estF}
Suppose Assumption \ref{assumption1} and \ref{assumption2} both hold. Let $U_{\varepsilon, r}$ denote the localization in Lemma \ref{lemmaZ} and let  $\mathcal{U}$ and $\bar \Sigma$ be the matrices in Lemma \ref{Gamma-prop}. Set
$$
\a= \mathcal{U}\bar{\Sigma} \quad\mbox{and}\quad \widehat X_\delta=\a^{-1}  (X_\d-x_0-b(0,x_0)\d).
$$
Then there exist $C\in \cal{C}$, $\d_*,r\in 1/\cal{C}$ such that for $\d\leq \d_*, \, |z|\leq r/2$,
\be{estlocF}
p_{\widehat{X}_\d,U_{\ve,r}}(z)\geq \frac 1C,
\ee
$p_{\widehat{X}_\d,U_{\ve,r}}$ denoting the density of $\widehat{X}_\d$ with respect to the measure  $\PR_{U_{\ve,r}}$.
\el
\bpr
We set $\widehat{Z}_\delta=\alpha^{-1}Z_\delta$ and we use Proposition \ref{generalbounds}, with the localization $U=U_{\ve,r}$, applied to $F=\widehat{X}_\delta$ and $G=\widehat Z_\delta$. Recall that the requests in (1) of Proposition \ref{generalbounds} involve several quantities: the lowest singular value (that in this case coincides with the lowest eigenvalue) $\lambda_*(\gamma_{\widehat{X}_\delta})$ and $\lambda_*(\gamma_{\widehat{Z}_\delta})$ of the Malliavin covariance matrix of $\widehat{X}_\delta$ and $\widehat{Z}_\delta$ respectively, as well as $m_{U_{\ve,r}}(p)$ in \eqref{defmU}, the Sobolev-Malliavin norms $\|\widehat{X}_\delta\|_{2,p,U_{\ve,r}}$, $\|\widehat{Z}_\delta\|_{2,p,U_{\ve,r}}$, and $\|\widehat{X}_\delta-\widehat{Z}_\delta\|_{2,p,U_{\ve,r}}=\|  \a^{-1} R_\d\|_{2,p,U_{\ve,r}}$. First of all, by using Assumption \ref{assumption1}, one easily gets
\begin{equation}\label{est1}
\|\a^{-1} R_\d\|_{2,p}\leq C\d^{-1} \d^{3/2}=C\sqrt \d
\quad\mbox{and}\quad
\|\widehat{X}_\delta\|_{2,p}+\|\widehat Z_\delta\|_{2,p}
\leq C,\quad \exists C\in\mathcal{C}.
\end{equation}
We now check that $m_{U_{\ve,r}}(p)<\infty$ for every $p$. Standard computations and \eqref{Sup1} give, for every $p$,
\[
\| 1/\det Q \|_{2,p}+
\| \,|Q|_l \|_{2,p}+
\|q(B)\|_{2,p}+
\|\Th \|_{2,p}\leq C,
\]
so we can apply \eqref{normlocalization} and conclude
\be{estimateU}
m_{U_{\ve,r}}(p) \leq  C \in \cal{C}.
\ee
We now study the lower eigenvalue of the Malliavin covariance matrix of $\widehat{Z}_\delta$. From the definition of $\widehat{Z}_\delta$, we have
\be{ZG}
\widetilde{Z}=
\mathcal{V} \left(
\begin{array}{c}
\a^{-1} Z_\d \\
\underline{\mathcal{U}}^T\underline{\Gamma}_\delta\Theta
\end{array}
\right)
=
\mathcal{V} \left(
\begin{array}{c}
\widehat{Z}_\delta \\
\underline{\mathcal{U}}^T\underline{\Gamma}_\delta\Theta
\end{array}
\right),
\ee
(see the proof of Lemma \ref{Gamma-prop} for the definition of $\underline{\Gamma}_\delta$).
As an immediate consequence, one has $ \l_*(\g_{\widehat{Z}_\delta}) \geq
\l_*(\g_{\widetilde{Z}})$, and it suffices to study the lower eigenvalue of $\widetilde{Z}$. By using \eqref{New11}, we have
\[
\begin{split}
\langle\g_{\widetilde{Z}} \xi, \xi \rangle
& = \sum_{i=1}^d \int_0^\d \langle D^i_s \widetilde{Z},\xi \rangle^2
= \sum_{i=1}^d \int_{s_{i-1}(\d)}^{s_i(\d)} \langle D^i_s \widetilde{Z},\xi \rangle^2
= \sum_{i=1}^d \int_{s_{i-1}(\d)}^{s_i(\d)} \langle D^i_s (\Theta+\widetilde \eta(\Theta)),\xi \rangle^2\\
& \geq \sum_{i=1}^d \int_{s_{i-1}(\d)}^{s_i(\d)} \Big(\frac{1}{2}\langle D^i_s \Th,\xi \rangle^2 - \langle D^i_s \eta(\Th),\xi \rangle^2\Big) ds\\
&=S_1+S_2.
\end{split}
\]
We write
\begin{align*}
&S_1=\sum_{i=1}^d \int_{s_{i-1}(\d)}^{s_i(\d)} \frac{1}{2}\langle D^i_s \Th,\xi \rangle^2\geq \frac{\l_*(Q)}{2}\\
&S_2=\sum_{i=1}^d \int_{s_{i-1}(\d)}^{s_i(\d)} \langle \nabla\eta(\Th) D^i_s \Th,\xi \rangle^2 ds
= \sum_{i=1}^d \int_{s_{i-1}(\d)}^{s_i(\d)} \langle D^i_s \Th, \nabla\eta(\Th)^T \xi \rangle^2 ds
\leq \l^*(Q) |\nabla\eta(\Th)|^2 |\xi|^2,
\end{align*}
so that
\[
\begin{split}
\l_*(\g_{\widehat{Z}_\delta}) \geq
\l_*(\g_{\widetilde{Z}}) \geq \l_*(Q) \left(\frac{1}{2} - \frac{\l^*(Q)}{\l_*(Q)} |\nabla\eta(\Th)|^2\right).
\end{split}
\]
On $\{\widetilde{U}_\varepsilon\neq 0\}$, we have already proved in Lemma \ref{lemmaZ} that $c_*(\eta,\Th)\leq \frac{\sqrt{\l_*(Q)/\l^*(Q)}}{2m}$. Since $|\nabla\eta(\Th)| \leq m c_*(\eta,\Th)$, we obtain
\[
|\nabla\eta(\Th)| \leq \frac{1}{2}\sqrt{\frac{\l_*(Q)}{\l^*(Q)}},
\]
and therefore $\l_*(\g_{\widehat{Z}_\delta})\geq 4 \l_*(\g_{\widetilde{Z}})\geq \l_*(Q)\geq \ve^{3m\r}$, which implies that $\E_{U_{\ve,r}}(\lambda_*(\widehat{Z}_\delta)^{-p})<\infty$ for every $p$.
Let us study the lowest eigenvalue of $\gamma_{\widehat{X}_\delta}$. We use here some results from next Section \ref{upper}, namely Lemma \ref{gammaF}. There, we actually prove the desired bound for the Malliavin covariance matrix of $\a^{-1}(X_\d-x_0)$. Here we are considering $\widehat{X}_\delta=\a^{-1}(X_\d-x_0-b(0,x_0)\d)$, but their Malliavin covariance matrix is the clearly the same. Then, Lemma \ref{gammaF} gives that
$\E(\lambda_*(\gamma_{\widehat{X}_\delta})^{-p})<\infty$ for every $p$.

So, we have proved that all the requests in Proposition \ref {generalbounds} hold. Then, we can apply \eqref{generallowerbound} and we get
$$
p_{{\widehat{X}_\delta},U_{\ve,r}}(z)
\geq p_{{\widehat{Z}_\delta},U_{\ve,r}}(z)
-C'\sqrt\delta
$$
with $C'\in \mathcal{C}$. Now, from \eqref{ZG} and \eqref{estZ}, with a simple change of variables, we get
\be{boundG}
p_{{\widehat{Z}_\delta},U_{\ve,r}}(z) \geq
\frac{1}{C}, \quad \mbox{for}\quad |z|\leq \frac{r}{2}.
\ee
We can assert the existence of $\delta_*\in 1/\mathcal{C}$ and  $C\in \mathcal{C}$ such that for all $\delta\leq \delta_*$
$$
p_{{\widehat{X}_\delta},U_{\ve,r}}(z)
\geq \frac 1{2C},
$$
and the statement follows.
\epr

We are now ready for the proof of the lower bound:

\bt{lower-th}
Let Assumption \ref{assumption1} and \ref{assumption2} hold. Let $p_{X_t}$ denote the density of $X_t$, $t>0$. Then there exist positive constants $r,\d_*,C$ such that for every $\d\leq \d_*$ and for every $y$ such that $|y- x_0-b(0,x_0)\d|_{A_\d(0,x_0)}\leq r$,
\[
p_{X_\d}(y) \geq
\frac{1}{C \d^{n-\frac{\dim\langle \s(0,x_0) \rangle}{2} }},
\]
$\dim\langle \s(0,x_0) \rangle$ denoting the dimension of the vector space spanned by $\sigma_1(0,x_0),\ldots,\sigma_d(0,x_0)$.
Here, $C\in\mathcal{C}$ and  $r,\delta_\ast\in 1/\mathcal{C}$.
\et
\bpr
We take the same $\d_*,r$ as in Lemma \ref{estF} and let ${\widehat{X}_\delta}$ denotes the r.v. handled in Lemma \ref{estF}. By construction, we have $X_\delta=x_0+b(0,x_0)+\alpha {\widehat{X}_\delta}$, so by applying Lemma \ref{estF} we get
\[
\begin{split}
\E(f(X_\d)) & \geq \E_{U_{\ve,r}}(f(X_\d)) = \E_{U_{\ve,r}}(f(x_0 + b(0,x_0)\d+ \a \widehat{X}_\delta))\\
&= \int f(x_0 +b(0,x_0)\d+ \a z) p_{\widehat{X}_\delta,U_{\ve,r}}(z) dz\\
&\geq \frac{1}{C} \int_{\{|z|\leq r/2\}} f(x_0 +b(0,x_0)\d+ \a z) dz\\
&\geq \frac{1}{C|\det \a |} \int_{|y|_{\a}\leq r/2}
f(x_0 +b(0,x_0)\d+ \,y) dy\\
\end{split}
\]
From \eqref{Adlucia} and the Cauchy-Binet formula we obtain
\be{CB}
\frac{1}{C} \d^{n-\frac{\dim\langle \s \rangle}{2}} \leq
\sqrt{|\det A_\d A_\d^T |} =\det(\a)\leq
C \d^{n-\frac{\dim\langle \s \rangle}{2}}
\ee
and the statement follows.
\epr

\begin{remark}\label{classC}
We observe that if the diffusion coefficients are bounded, that is Assumption \ref{assumption1bis} holds, then the class $\mathcal{C}$ in \eqref{CC} of the constants can be replaced by
$$
\mathcal{D}_0=\Big\{C>0\,:\,C=c \Big(\frac{\kappa}{\lambda(0,x_0)}\Big)^{q},\ \exists\ c,q>0\Big\}
$$
and, as before, $1/\mathcal{D}_0=\{C>0\,:\,1/C\in\mathcal{D}_0\}$. This is because in the estimates for $\|\widehat{X}_\delta-\widehat Z_\delta\|_{2,p}$ and $\|\widehat{X}_\delta\|_{2,p}$ one does not need any more to use the Gronwall's Lemma but it suffices to use the boundedness of the coefficients and the Burkholder inequality.
\end{remark}

\section{Upper bound}\label{upper}

We study here the upper bound for the density of $X_\d$.

\subsection{Rescaling of the diffusion}
As for the lower bound, we again scale $X_\d$. We recall the results and the notation in Lemma \ref{Gamma-prop} and we define the change of variable
\be{Ta}
T_\a:\R^n \rightarrow \R^n , \quad\quad T_\a(y)=\a^{-1} y, \quad \mbox{where}\quad \a=\mathcal{U} \bar{\Sigma}
\ee
and its adjoint $T_\a^*(v)=\a^{-1,T} v$. Note that $\a$ is a $n\times n$ matrix. We write $A_{\d,j}$, for $j=1,\dots, m$, for the columns of $A_\d$ (which can be $\sqrt{\d}\s_i$ or $\d[\s_i,\s_p]$). The following properties hold:
\begin{lemma}
Let $T_\alpha$ be defined in (\ref{Ta}). Then one has:
\begin{align}
& |y|_{A_\d}=|T_\a y|=|y|_\a,\quad \forall y\in \R^n, \quad\mbox{ and }\det \a= \sqrt{\det A_\d A_\d^T} \label{p1}\\
& \forall v\in \R^n \mbox{ with }|v|=1,\quad \exists j=1,\dots, m : \quad |T_\a^*v \cdot A_{\d,j}|\geq \frac{1}{m} \label{p2}\\
& \forall j=1,\dots,d, \quad \sqrt{\d}|T_\a\s_j| \leq 1 \label{p3}
\end{align}
\end{lemma}
\bpr{}
\eqref{p1} follows easily from $\a=\mathcal{U} \bar{\Sigma}$ and the definition \eqref{Not4} of $|\cdot|_{M}$. Now, $(T_\a^*v)^T A_\d = v^T \a^{-1} A_\d = [v^T 0] \mathcal{V}^T$. So $|(T_\a^*v)^T A_\d|=|[v^T 0] \mathcal{V}^T|=1$. Recall that $A_{\d,j}$ are the columns of $A_\d$, therefore $\exists j=1,\dots, m : \quad | (T_\a^*v)^T A_{\d,j}|\geq \frac{1}{m}$, which is equivalent to \eqref{p2}. Moreover, $T_\a A_\d=[\mathrm{Id}_n\, 0]\mathcal{V}^T$. This easily implies that $\forall i=1,\dots, m,\quad  | T_\a A_{\d,i}|\leq 1$. For $A_{\d,i}=\s_j(0,x_0)\sqrt{\d}$ we have \eqref{p3}.
\epr
We define now
\begin{equation}\label{F}
F=\a^{-1} (X_\d-x_0)=T_\a (X_\d-x_0).
\end{equation}
As for the lower bound, we first estimate the density of $F$, using the results in Appendix \ref{app-mall} (specifically, \eqref{generalupperbound} in Proposition \ref{generalbounds}), and then recover the estimates for the density of $X_\d$ via a change of variable.

\subsection{Malliavin Covariance Matrix}\label{sectionnorris}

Let $F$ be as in \eqref{F}. To prove the upper bound for its density $p_F$ we need a  quantitative control on the Malliavin covariance matrix $\g_F$ of $F$. We start with some preliminary results.

The following lemma is a slight modification of Lemma 2.3.1. in \cite{Nualart:06}.
\bl{lemma:matrixmoment}
Let $\g$ be a symmetric nonnegative definite $n\times n$ random matrix. Denoting $|\g|=\sum_{1\leq i,j \leq n} |\g^{i,j}|^2)^{1/2}$, we assume that, for $p\geq 2$, $\E|\g|^{p+1} < \infty$, and that there exists $\ve_0>0$ such that for $\ve \leq \ve_0$,
\[
\sup_{|\xi|=1} \PR [\langle\g \xi,\xi \rangle<\ve]\leq \ve^{p+2n}
\]
Then there exists a constant $C$ depending only on the dimension $n$ such that
\[
\E \l_*(\g)^{-p} \leq C \E |\g|^{1+p} \ve_0^{-p}
\]
\el
We also need the following technical result.
\bl{covT}
Let $\d\in(0,1]$ and let $a_t, b_t$, $t\in [0,\d]$ be stochastic processes which are a.s. increasing.  Assume that $b_0=0$. Suppose for fixed $p\geq 1$ and for all $ t\in[0,\d]$ one has
\[
\E[b_t^p]\leq C_p t^{2p}\quad\mbox{ and }\quad a_t\geq \frac{t-b_t}{\d}.
\]
Then for all $\ve>0$
\[
\PR( a_\d \leq \ve)\leq
4^{p}C_p \ve^{p}.
\]
\el
\bpr
Set
\[
S_\ve = \inf  \left\{ s \geq 0: b_s \geq \frac{\d\ve}{2} \right\} \wedge \d,
\]
Remark that for any $p>0$
\[
\PR(S_\ve < \d\ve)
=
\PR\left(
b_{\d\ve}^p
\geq \left(\frac{\d\ve}{2}\right)^p
\right)
\leq 2^p \frac{\E b_{\d\ve}^p}
{(\d\ve)^{p}}
\leq 2^p C_p (\d \ve)^p .
\]
On the other hand, on $S_\ve\geq \d\ve$,
\[
a_{S_\ve} \geq a_{\d\ve} \geq \frac{\d\ve-\d\ve/2 }{\d} \geq
\ve/2.
\]
Therefore
\[
\PR(a_\d < \ve/2) \leq
\PR(a_\d < \ve/2, S_\ve<\delta\ve)+\PR(a_\d < \ve/2, S_\ve\geq \delta\ve)
\leq
\PR(S_\ve < \d\ve)\leq 2^p C_p\ve^p.
\]
This implies that $\PR(a_\d < \ve) \leq 4^{p}C_p \ve^{p}$.
\epr
The following Lemma \ref{norris} is a refinement of what was proved by Norris in \cite{Norris}, in the sense that we take care of the same quantities, but handling more carefully the dependence on the final time $t_0$. This is a key estimate for the proof of next Theorem \ref{upper-th}.

\bl{norris}
Suppose $u(t)=(u_1(t),\dots, u_d(t))$ and $a(t)$ are a.s. continuous and adapted processes such that for some $p\geq 1$, $C>0$ and for every $t_0\leq 1$ one has
\be{int_norr}
\E \left[\sup_{0\leq s \leq t_0} |u_s|^p \right]\leq \frac{C}{t_0^p} ,\quad\quad \E \left[\sup_{0\leq s \leq t_0} |a_s|^p \right]\leq \frac{C}{t_0^p}.
\ee
Set
\[
Y(t)=y+\int_0^t a(s) ds + \sum_{k=1}^d \int_0^t u_k(s) dW^k_s.
\]
Then, for any $q>4$ and $r>0$ such that $6r+4<q$,  there exists $\ve_0(q,r,p)$ such that  for every $t_0 \leq 1$ and $\ve\leq \ve_0(q,r,p)$ one has
\[
\PR\left\{\int_0^{t_0} Y_t^2 dt <\ve^q,\, \int_0^{t_0} |u(t)|^2 dt \geq \frac{6\ve}{t_0}
\right\}\leq (2^pC+1)\ve^{rp}.
\]
\el

\bpr{}
Set $\th_t=|a_t|+|u_t|$, and
\[
\tau=\inf\left\{ s\geq 0: \sup_{0\leq u\leq s} \th_u >\frac{\ve^{-r}}{t_0} \right\} \wedge t_0.
\]
We have
\[
\PR\left\{\int_0^{t_0} Y_t^2 dt <\ve^q,\, \int_0^{t_0} |u(t)|^2 dt \geq \frac{\ve}{t_0}
\right\}
\leq A_1+A_2.
\]
where $A_1=\PR[\tau<t_0]$ and
\[
A_2= \PR\left\{
\int_0^{t_0} Y_t^2 dt<\ve^q,\int_0^{t_0} |u_t|^2dt\geq \frac{\ve}{t_0}, \tau=t_0
\right\}
\]
An upper bound for $A_1$ easily follows from \eqref{int_norr}. Indeed
\[
\PR[\tau<t_0]\leq \PR\left[\sup_{0\leq s\leq t_0} \th_u >\frac{\ve^{-r}}{t_0}\right]
          \leq t_0^p \ve^{rp}\E\left[\sup_{0\leq s\leq t_0} \th_s^p\right]
			 \leq 2^p C\ve^{rp}.
\]
for $\ve\leq \ve_0$. To estimate $A_2$ we introduce
\begin{align*}
&N_t=\int_0^t Y_s \sum_{k=1}^d u^k_s dW^k_s\quad\mbox{and}\\
&B= \left\{\langle N\rangle_\tau  <\r,\sup_{0\leq s\leq \tau }|N_s|\geq \d \right\},\quad\mbox{with}
\quad \d=\frac{\ve^{2r+2}}{t_0}  \quad\mbox{and}\quad \r=\frac{\ve^{q-2r}}{t_0^2}.
\end{align*}
By the exponential martingale inequality,
\[
\PR(B)\leq \exp (\frac{-\d^2}{2\r} )\leq \exp\left( -\ve^{6r+4-q} \right).
\]
So, in order to conclude the proof, it suffices to show that
\be{incl}
\left\{
\int_0^{t_0} Y_t^2 dt<\ve^q,\int_0^{t_0} |u_t|^2 dt\geq \frac{6\ve}{t_0}, \tau =t_0
\right\}\subset B,
\ee
We suppose $\o \notin B$, $\int_0^{t_0} Y_t^2 dt<\ve^q$ and $\tau=t_0$ and show $\int_0^{t_0} |u_t|^2dt < 6\ve/t_0$. With these assumptions,
\[
\langle N \rangle_\tau =\int_0^{\tau } Y_t^2 |u_t|^2 dt
\leq \int_0^{t_0} Y_t^2 dt\,\sup_{0\leq t\leq \tau }|u_t|^2 <\frac{\ve^{q-2r}}{t_0^2}=\r.
\]
So, for $\o\notin B$ then one necessarily has $\sup_{0\leq t\leq \tau } |\int_0^{t} Y_s \sum_{k=1}^d u^k_s dW^k_s|<\d=\ve^{2r+2}/t_0$. From $6r+4\leq q$, if $\tau =t_0$ then
\begin{align*}
\sup_{0\leq t\leq \tau }\left|\int_0^t Y_s a_s ds\right|
&\leq \left( t_0 \int_0^\tau  Y_s^2 a_s^2 ds\right)^{1/2}
\leq t_0\Big(\int_0^{t_0}Y_s^2ds\sup_{0\leq s\leq \tau }|a_s|^2\Big)^{1/2}\\
&\leq t_0\Big(\ve^q\,\frac {\ve^{-2r}}{t_0^2}\Big)^{1/2}
\leq
\frac{\ve^{2r+2}}{t_0}.
\end{align*}
Thus
\[
\sup_{0\leq t\leq \tau } \left|\int_0^t Y_s dY_s \right|\leq
\sup_{0\leq t\leq \tau } \left|\int_0^t Y_s a_s ds+\int_0^t Y_s u_s dW_s \right|
\leq \frac{2 \ve^{2r+2}}{t_0}.
\]
By It\^o's formula, $Y_t^2=y^2+2\int_0^t Y_s dY_s+\langle M\rangle_t$ with $\langle M\rangle_t=\int_0^t |u_s|^2ds$. So, recalling that $q>2r+2$,
\[
\int_0^\tau  \langle M\rangle_tdt   =
\int_0^\tau  Y_t^2 dt - \tau  y^2 - 2\int_0^\tau  \int_0^t Y_s dY_s dt \\
 < \ve^q +4 \frac{\ve^{2r+2}}{t_0}<5\frac{\ve^{2r+2}}{t_0}.
\]
Since $t\mapsto \langle M\rangle_t$ is non negative and increasing, for $0<\g <\tau $ we have
\[
\g \langle M \rangle_{\tau -\g}\leq \int_{\tau -\gamma}^\tau \langle M\rangle_tdt
\leq 5\frac{\ve^{2r+2}}{t_0}.
\]
Using also the fact that
\[
\langle M\rangle_\tau -\langle M\rangle_{\tau -\g}=\int_{\tau -\g}^\tau  |u_s|^2 ds \leq \g \frac{\ve^{-2r}}{t_0^2},
\]
we have
\[
\langle M \rangle_\tau  <\frac{5\ve^{2r+2}}{\g}+ \g \frac{\ve^{-2r}}{t_0^2}.
\]
With $\g=t_0 \ve^{2r+1}$, this gives $\int_0^{t_0} |u_s|^2 ds =\langle M \rangle_\tau  <\frac{6\ve}{t_0}$.
\epr

We are now ready to prove the non degeneracy of the Malliavin covariance matrix. More precisely, we prove a quantitative version of this property: the $L^p$ norm of the inverse of the Malliavin covariance matrix of $F$ is upper bounded by a constant in $\mathcal{C}$, $\mathcal{C}$ being defined in \eqref{CC}.

\bl{gammaF}
Let $\a$, $T_\alpha$  and $F=T_\a (X_\d-x_0)$ be defined as in \eqref{Ta} and \eqref{F}. Let $\gamma_F$ denote the Malliavin covariance matrix of $F$. Then  for any $p>1$ there exists $C\in\cal{C}$ such that, for $\d\leq 1$, $\E |\l_*(\g_F)|^{-p}\leq C$.
\el
\bpr
We need a bound for the moments of the inverse of
\[
\g_F= \sum_{k=1}^d \int_0^\d D^k_s F D^k_s F^T ds.
\]
Following \cite{Nualart:06} we define the tangent flow $Y$ of $X$ as the derivative with respect to the initial condition of $X$: $Y_t:=\partial_x X_t$. We also denote its inverse $Z_t=Y_t^{-1}$. Then one has (remark that the equations we consider for $X$, $Y$ and $Z$ are all in Stratonovich form):
\be{defZY}
\begin{split}
Y_t&=\mathrm{Id} + \sum_{k=1}^d\int_0^t \nabla_x \s_k(s,X_s) Y_s \circ dW^k_s+\int_0^t \nabla_x b(s,X_s) Y_s ds\\
Z_t&=\mathrm{Id} - \sum_{k=1}^d\int_0^t Z_s \nabla_x \s_k(s,X_s) \circ dW^k_s-\int_0^t Z_s \nabla_x b(s,X_s)ds,
\end{split}
\ee
where $\nabla_x \s_k$ and $\nabla_x b$ are the Jacobian matrix with respect to the space variable. It holds
\[
D_s X_\d=Y_\d Z_s \s(s,X_s),\quad s<\delta.
\]
By applying It\^o's formula we have the following representation, for $\phi\in C^{1,2}$:
\be{devZphim}
\begin{split}
Z_t \phi(t,X_t)&=
\phi(0,x_0)+\int_0^t Z_s \sum_{k=1}^d [\s_k,\phi](s,X_s) dW_s^k\\
&+
\int_0^t Z_s \left\{ [b,\phi]+ \frac{1}{2} \sum_{k=1}^d [\s_k,[\s_k,\phi]]+ \frac{\partial \phi}{\partial s} \right\} (s,X_s)\,ds
\end{split}
\ee
(details are given in \cite{Nualart:06}, remark that in the r.h.s. above we are taking into account an It\^o integral). We now compute
\[
D_s F = \a^{-1} D_s X_\d
=\a^{-1} Y_\d Z_s \s(s,X_s)
=\a^{-1} Y_\d \a \a^{-1} Z_s \s(s,X_s)
\]
so
\[
\g_F = \a^{-1} Y_\d \a \, \bar{\g}_F \, (\a^{-1} Y_\d \a)^T \quad \mbox{ where }\quad
\bar{\g}_F =
\a^{-1}
\int_0^\d
Z_s \s(s,X_s)
\s(s,X_s)^T Z_s^T ds \,
\a^{-1,T},
\]
and
\[
\g^{-1}_F = (\a^{-1} Y_\d \a)^{-1,T} \, \bar{\g}_F^{-1} \, (\a^{-1} Y_\d \a)^{-1}.
\]
Now,
\[
(\a^{-1} Y_\d \a)^{-1} =
\a^{-1} Z_\d \a =
\mathrm{Id}_n
+ \a^{-1} (Z_\d-\mathrm{Id}_n) \a
\]
Using the fact that $\l^*(\cdot)$ is a norm on the set of matrices, and that for two $n\times n$ matrices $A,B$,  $\l^*(AB)\leq n \l^*(A) \l^*(B) $, we have
\[
\l_*\big(\g_F\big)^{-1}=
\l^*\big(\g_F^{-1}\big) \leq n^2 \l^*\big(\bar{\g}_F^{-1}\big) \,\l^*\big( (\a^{-1} Y_\d \a)^{-1}\big)^2
\]
and
\[
\l^*\big((\a^{-1} Y_\d \a)^{-1}\big)\leq 1+ n^2\l^*(\a^{-1}) \l^*(Z_\d-\mathrm{Id}_n) \l^*(\a).
\]
Standard estimates (see also \eqref{defZY}) give $\l^*(Z_\d-\mathrm{Id}_n) \leq C_1\sqrt{\d}$ for some $C_1\in \cal{C}$. Moreover
\[
\begin{split}
& \l^*(\a)=\l^*(A_\d)=\l^*(A \cal{D}_\d) \leq n \l^*(A) \l^*(\cal{D}_\d)\leq C_2 \sqrt{\d}, \quad C_2\in \cal{C}\\
&\l^*(\a^{-1})\leq \frac{1}{\l_*(A \cal{D}_\d)} \leq \frac{C_3}{\d} , \quad C_3\in \cal{C}
\end{split}
\]
and so for all $q>1$ exists $C \in \cal{C}$ such that
\[
\E \l^*\left((\a^{-1} Y_\d \a)^{-1}\right)^{q} \leq C
\]
We now need to estimate the reduced matrix, i.e. prove
that for all $q>1$ exists $C \in \cal{C}$ such that
\be{estred}
\E \l^*(\bar{\g}_F^{-1})^q=
\E \l_*(\bar{\g}_F)^{-q} \leq C
\ee
We show now that for any $p>0$, $\sup_{|v|=1} \PR\left( \langle \bar{\g}_F v,v \rangle\leq \ve\right) \leq \ve^p$, for $\d\leq 1$ for $\ve\leq \ve_0\in 1/\cal{C}$ not depending on $\d$. Together with lemma \ref{lemma:matrixmoment} this implies \eqref{estred}.

Denote $\xi = T_\a^* v=\alpha^{-1,T}v$. From \eqref{p2} and the definition \eqref{Adlucia} of $A_\d$ we have two possible cases:
A) $|\xi \cdot \s_j(0,x_0)| \geq \frac{1}{m\d^{1/2}}$ for some $j=1,\dots, d$, or B)
$|\xi \cdot [\s_j,\s_l](0,x_0)|\geq \frac{1}{m\d}$ for some $j,l=1,\dots, d,\, j\neq l$. Moreover
\be{changevariablematrix}
\a \bar{\g}_F \a^T
=\int_0^\d
Z_s \s(s,X_s)
\s(s,X_s)^T Z_s^T ds.
\ee
Therefore, with $\xi = T_\a^* v$, we have for any $q>1$
\[
\begin{split}
\PR(\langle \bar{\g}_F v,v \rangle\leq \ve^q) &=
\PR\left(\xi^T \int_0^\d Z_s \s(s,X_s)\s(s,X_s)^T Z_s^T ds\, \xi \leq \ve^q \right) \\
&=
\PR\left( \sum_{i=1}^d \int_0^\d |\xi^T Z_s \s_i(s,X_s) |^2 ds \leq \ve^q\right)
\end{split}
\]
We decompose this probability:
\[
\begin{split}
\PR(\langle &\bar{\g}_F v,v \rangle\leq \ve^q)
=
\PR\left( \sum_{i=1}^d \int_0^\d |\xi^T Z_t \s_i(t,X_t) |^2 dt \leq \ve^q\right)\\
&\leq
\PR\left( \sum_{i=1}^d \int_0^\d |\xi^T Z_t \s_i(t,X_t) |^2 dt \leq \ve^q,\sum_{i,k=1}^d \int_0^\d |\xi^T Z_t [\s_i,\s_k](t,X_t) |^2 dt \leq \frac{\ve}{\d}\right)\\
&+
\PR\left( \sum_{i=1}^d \int_0^\d |\xi^T Z_t \s_i(t,X_t) |^2 dt \leq \ve^q, \sum_{i,k=1}^d \int_0^\d |\xi^T Z_t [\s_i,\s_k](t,X_t) |^2 dt > \frac{\ve}{\d} \right)\\
&=:I_1+I_2
\end{split}
\]
To estimate $I_1$ we distinguish the two cases A) and B) above.

Case A): $|\xi \cdot \s_j(0,x_0)| \geq \frac{1}{m\d^{1/2}}$ for some $j=1,\dots, d$. We fix this $j$. Then,
\begin{align*}
&I_1
\leq
\PR\Big( \int_0^\d |\xi^T Z_t \s_j(t,X_t) |^2 dt \leq \ve^q,
\int_0^{\d} \Big|\xi^T \sum_{k=1}^d Z_t [\s_k,\s_j](t,X_t)\Big|^2 dt < \frac{ \ve}{\d} \Big) \\
&\leq
\PR\Big( \int_0^\d |\xi^T Z_t \s_j(t,X_t) |^2 dt \leq \ve^q,
\sup_{0\leq t\leq \d} \Big| \int_0^t \xi^T \sum_{k=1}^d Z_s [\s_k,\s_j](s,X_s) dW^k_s \Big|^2 < \frac{1}{12 m^2 \d} \Big)\\
&+\PR\Big(\sup_{0\leq t\leq \d}  |\int_0^t \xi^T \sum_{k=1}^d Z_s [\s_k,\s_j](s,X_s) dW_s^k |^2 \geq \frac{1}{12 m^2 \d},
\int_0^\d |\xi^T \sum_{k=1}^d Z_t [\s_k,\s_j](t,X_t)|^2 dt <\frac{ \ve}{\d}  \Big)
\end{align*}
Set $u_s=(\xi^T Z_s [\s_k,\s_j](s,X_s))_{k=1,\dots, d}$. From the exponential martingale inequality we have
\be{mart}
\begin{split}
&\PR\Big(\sup_{0\leq t\leq\d }| \sum_{k=1}^d \int_0^t u^k_s dW^k_s |^2 \geq  \frac{1}{12 m^2 \d},\quad \int_0^\d |u_t|^2 dt < \frac{\ve}{\d}  \Big) \\
&\quad \leq
2\exp\left( -\frac{1}{ 12 m^2\d} \times \frac{\d}{ 2\ve} \right)
=2\exp\left( -\frac{1}{ 24  m^2 \ve} \right)<\ve^p,
\end{split}
\ee
the latter inequality holding for every $p>1$ and $\ve\leq \ve_0$.
We now define
$$
D:=\Big\{
\sup_{0\leq t\leq \d} |\int_0^t \xi^T \sum_{k=1}^d Z_s [\s_k,\s_j](s,X_s) dW_s^k |^2 < \frac{1}{12 m^2 \d}\Big\}
$$
and prove
\[
\PR\left( \left\{\int_0^\d |\xi^T Z_t \s_j(t,X_t) |^2 dt \leq \frac{\ve^q}{4m^2}\right\}
\cap D \right)\leq \ve^p
\]
which is equivalent to the desired estimate $\PR\left( \left\{\int_0^\d |\xi^T Z_t \s_j(t,X_t) |^2 dt \leq \ve^q\right\} \cap D \right)\leq \ve^p$.
From representation \eqref{devZphim}, for $\phi=\s_j$ we find
\[
Z_t \s_j(t,X_t)=  \s_j(0,x_0)+ \int_0^t \sum_{k=1}^d Z_s [\s_k,\s_j](s,X_s)dW^k_s
+R_t,
\]
with
\[
R_t= \int_0^t Z_s \left\{ [b,\s_j]+ \frac{1}{2} \sum_{k=1}^d [\s_k,[\s_k,\s_j ]]+\frac{\partial\s_j}{\partial s}\right\}(s,X_s)\,ds.
\]
From $(a+b+c)^2\geq a^2/3 - b^2-c^2$ and $|\xi\cdot \s_j(0,x_0)|\geq \frac{1}{m\d^{1/2}}$, for $\bar{t}\leq \d$ we can write
\[
\begin{split}
& \int_0^{\bar{t}} |\xi^T Z_t \s_j(t,X_t) |^2 dt \\
&\quad\geq
\frac{\bar{t} |\xi^T \s_j(0,x_0) |^2}{3} -
\int_0^{\bar{t}} | \sum_{k=1}^d  \int_0^t  \xi^T Z_s [\s_k,\s_j](s,X_s)dW^k_s|^2 dt
-\int_0^{\bar{t}} | \xi^T R_t|^2 dt \\
&\quad \geq
\frac{\bar{t}}{3 \d m^2}
- \int_0^{\bar{t}} | \sum_{k=1}^d \int_0^t \xi^T Z_s [\s_k,\s_j](s,X_s)dW^k_s|^2 dt
-\int_0^{\bar{t}} |\xi^T R_t|^2 dt.
\end{split}
\]
On the set $D$ one has
\[
\int_0^{\bar{t}} | \sum_{k=1}^d \int_0^t \xi^T Z_s [\s_k,\s_j](s,X_s)dW^k_s|^2 dt
\leq {\bar{t}}\frac{1}{12 m^2\d},
\]
so
\[
\begin{split}
\int_0^{\bar{t}} |\xi^T Z_t \s_j(t,X_t) |^2 dt \geq
\frac{\bar{t}}{4 m^2 \d} -
\int_0^{\bar{t}} | \xi^T R_t|^2 dt,
\end{split}
\]
that we rewrite as
\be{hplab}
\begin{split}
4 m^2 \int_0^{\bar{t}} |\xi^T Z_t \s_j(t,X_t) |^2 dt \geq
\frac{\bar{t} -
4 m^2 \d \int_0^{\bar{t}} | \xi^T R_t|^2 dt }{\d}.
\end{split}
\ee
We now set
\[
\mbox{$\displaystyle a_{\bar{t}}= 4 m^2 \int_0^{\bar{t}} |\xi^T Z_t \s_j(t,X_t) |^2 dt$ on the set $D$ and $a_{\bar{t}}=\bar{t}/\d$ on the set $D^c$,}
\]
$D^c$ denoting the complement of $D$.
Standard computations, considering also $|\xi|=|T_\a^* v|\leq |v| C/\d = C/\d$, give
$\E (\int_0^{\bar{t}} | \xi^T R_t|^2 dt)^q \leq C {\bar{t}}^{3 q}/\d^{2 q},$
so
$\E ( 4 \d m^2 \int_0^{\bar{t}} | \xi^T R_t|^2 dt)^q \leq C {\bar{t}}^{2 q} $, for $C\in \cal{C}$ (recall also $\bar{t} \leq \d$).
This estimate and \eqref{hplab} allow us to apply lemma \ref{covT} with $a_{\bar{t}}$ defined above and
\[
b_{\bar{t}}= 4 \d m^2 \int_0^{\bar{t}} | \xi^T R_t|^2 dt.
\]
We find
\[
\left\{ \int_0^\d |\xi^T Z_t \s_j(t,X_t) |^2 dt \leq \frac{\ve^q}{4m^2}\right\}
\cap D =\{ a_\d \leq \ve^q \} \cap D
\]
and we have
\[
\PR\Big( \Big\{ \int_0^\d |\xi^T Z_t \s_j(t,X_t) |^2 dt \leq \frac{\ve^q}{4m^2}\Big\}
\cap D \Big)
=\PR( \{a_\d \leq \ve^q\} \cap D )
\leq
\PR( a_\d \leq \ve^q )\leq \ve^p.
\]
We the obtain $I_1 <\ve^p $ for any $p>1$, for $\d\leq 1$, $\ve\leq \ve_0$.

\smallskip

Case B) $|\xi \cdot [\s_j,\s_l](t,x_0)|\geq \frac{1}{m\d}$ for some $j,l=1\dots d,\, j\neq l$. In this case we write
\[
I_1\leq \PR\left(\int_0^\d | \xi^T Z_t [\s_j,\s_l](t,X_t) |^2 dt \leq \frac{\ve}{\d} \right)
\]
From representation \eqref{devZphim} with $\phi=[\s_j,\s_l]$ we find
\[
Z_t [\s_j,\s_l](t,X_t)=[\s_j,\s_l](0,x_0)+R_t,
\]
with
\be{repR2}
\begin{split}
R_t & = \int_0^t Z_s \sum_{k=1}^d [\s_k,[\s_j,\s_l] ](s,X_s) dW_s^k\\
&+
\int_0^t Z_s \left\{ [b,[\s_j,\s_l]]+ \frac{1}{2} \sum_{k=1}^d [\s_k,[\s_k,[\s_j,\s_l] ]]+\frac{\partial [\s_j,\s_l]}{\partial s} \right\}(s,X_s)\, ds.
\end{split}
\ee
From $(a+b)^2\geq a^2/2-b^2$ and $|\xi\cdot[\s_j,\s_l](0,x_0)|\geq \frac{1}{m \d}$, for $\bar{t}\leq \d$ we have
\be{inR}
\int_0^{\bar{t}} |\xi^T Z_t [\s_j,\s_l](t,X_t)|^2 dt
\geq  \frac{ \bar{t} |\xi^T [\s_j,\s_l](0,x_0)|^2 }{2} - \int_0^{\bar{t}} |\xi^T R_t|^2 dt
\geq  \frac{ \bar{t} }{2\d^2 m^2} - \int_0^{\bar{t}} |\xi^T R_t|^2 dt.
\ee
We apply lemma \ref{covT} with
\[
a_{\bar{t}}=2m^2 \d \int_0^{\bar{t}} | \xi^T Z_s [\s_j,\s_l](s,X_s) |^2 ds\quad\mbox{and}\quad
b_{\bar{t}}=2m^2 \d^2 \int_0^{\bar{t}} |\xi^T R_t|^2 dt
\]
Indeed from \eqref{repR2} and $|\xi|\leq C/\d$,
\[
\E |b_{\bar{t}}|^q \leq C \bar{t}^2
\]
and from \eqref{inR} we have $a_{\bar{t}}\geq \frac{\bar t-b_{\bar t}}\delta$. So, we find $I_1<\ve^p$, for $\d\leq 1$, $\ve\leq \ve_0$.

We estimate now
\[
I_2=\PR\left( \sum_{i=1}^d \int_0^\d |\xi^T Z_s \s_i(s,X_s) |^2 ds \leq \ve^q,
\sum_{i,j=1}^d \int_0^\d |\xi^T Z_s [\s_i,\s_j](s,X_s) |^2 ds > \frac{\ve}{\d} \right).
\]
By using again \eqref{devZphim}, we find
\[
\begin{split}
\xi^T Z_t \s_i(t,X_t)
& = \s_i(0,x_0) +\sum_{j=1}^d \int_0^t \xi^T Z_s [\s_j,\s_i](s,X_s) dW^j_s \\
& + \int_0^t
\xi^T Z_s \left\{ [b,\s_i]+ \frac{1}{2} \sum_{j=1}^n [\s_j,[\s_j,\s_i]]+\frac{\partial \s_i}{\partial s} \right\}(s,X_s) ds.
\end{split}
\]
For $t_0= \d$ and from the fact that $|\xi|\leq \frac{C}{\d}$, we have
\[
\begin{split}
&\E [\sup_{0\leq s \leq \d} |\xi^T Z_s [\s_j,\s_i](s,X_s)|^p ] \leq \frac{C}{\d^p}, \quad C\in \C, \quad \mbox{and}\\
&\E \left[\sup_{0\leq s \leq \d} |\xi^T Z_s \left\{[b,\s_i]+ \frac{1}{2} \sum_{j=1}^d [\s_j,[\s_j,\s_i]]\right\}(s,X_s)|^p \right] \leq \frac{C}{\d^p}, \quad C\in \C.
\end{split}
\]
Thus we can apply Lemma \ref{norris} and we get
\[
\PR\left( \sum_{i=1}^d \int_0^\d |\xi^T Z_s \s_i(s,X_s) |^2 ds \leq \ve^q \mbox{ and }
\sum_{i,j=1}^d \int_0^\d |\xi^T Z_s [\s_i,\s_j](s,X_s) |^2 ds > \frac{\ve}{\d} \right)\leq \ve^p
\]
for any $p>1$, $\d\leq 1$ for $\ve\leq \ve_0$. We have now both the estimates of $I_1$ and $I_2$, so we have
$\sup_{|v|=1} \PR(\langle \bar{\g}_F v,v \rangle\leq \ve^q)\leq \ve^p$ for
$p>1$, $\d\leq 1$ for $\ve\leq \ve_0$, and the statement holds.
\epr

\subsection{Upper bound for the density of $X_\d$}

\bt{upper-th}
Let Assumption \ref{assumption1} and \ref{assumption2} hold. Let $p_{X_t}$ denote the density of $X_t$, $t>0$. Then, for any $p>1$, there exists a positive constant $C\in \cal{C}$ such that for every $\d\leq 1$ and for every $y\in \R^n$
\[
p_{X_\d}(y)\leq  \frac{1}{\d^{n-\frac{\dim\langle \s(0,x_0) \rangle}{2}}
} \frac{C}{1+|y-x_0|_{A_\d(0,x_0)}^p}.
\]
Again, $\dim\langle \s(0,x_0) \rangle$ denotes the dimension of the vector space spanned by $\sigma_1(0,x_0),\ldots,$ $\sigma_d(0,x_0)$.
\et

\bpr
Set $F=T_\alpha(X_\delta-x_0)$. We apply estimate \eqref{generalupperbound}: there exist constants $p$ and $a$  depending only on the dimension $n$, such that
\[
p_F(z) \leq C \max\{1,\E |\l_* (\g_F)|^{-p} \|F\|_{2,p}  \} \PR(|F-z|<2)^{a}.
\]
We first show that $\|F\|_{2,p}\leq C\in \cal{C}$, as a consequence of Assumption \ref{assumption1}. We prove just that $\|F\|_p\leq C$ for every $p$, for the Malliavin derivatives the proof is heavier but analogous. We write
$$
F=T_\a \Big( \sum_{j=1}^{d}
\int_0^\d \s_{j}(t,X_{t})\circ dW_{t}^{j}+ \int_0^\d b(t,X_{t})dt
\Big)
=T_\a \Big(
\sum_{j=1}^{d} \s_{j}(0,x_0) W_\d^{j} +
B_\d \Big),
$$
where
\[
B_\d=
\sum_{j=1}^{d} \int_0^\d \big(\s_{j}(t,X_{t})-\s_{j}(0,x_0)\big) \circ dW_{t}^{j}+ \int_0^\d b(t,X_{t})dt.
\]
Therefore
\be{boundF}
|F|
\leq
\sum_{j=1}^{d} |T_\a \s_{j}(0,x_0) W_\d^{j} | +
| T_\a B_\d |.
\ee
\eqref{p3} implies $ |T_\a \s_{j}(0,x_0) W_\d^{j} | \leq C W_\d^{j}/\sqrt{\d}$, for $j=1,\dots,d$.
Moreover $| T_\a B_\d |\leq |B_\d|_{A_\d}\leq C |B_\d|/\d$. If assumption \ref{assumption1} holds we conclude that $\E |F|^p \leq C\in \cal{C}$.

As in \cite{BC14}, Remark 2.4, it is easy to reduce the estimate of $\PR(|F-z|<2)$ to the tail estimate of $F$, and then to use Markov inequality to relate the estimate of the tails to the  moments of $F$:
\be{tails}
\PR(|F-z|<2)\leq \PR(|F|>|z|/2) \leq C \frac{1\vee \E |F|^p}{1+|z|^p},\quad \forall z\in \R^n
\ee
Since, from Assumption \ref{assumption1}, all the moments of $F$ are bounded by constants in $\cal{C}$, we have that for any exponent $p>1$ this term decays faster than $|z|^{-p}$ for $|z|\rightarrow \infty$.

In Lemma \ref{gammaF} we have already proved that $\E |\l_*(\g_F)|^{-q}\leq C\in \cal{C}$, for $\d \leq 1$. We conclude that $p_F(z)\leq \frac{C}{1+|z|^p}$. The upper bound for the density of $X_\d$ comes from the simple change of variable
$y=x_0+\a z$. For a positive and bounded measurable function $f:\R^n\rightarrow \R$, we write
\[
\E f(X_\d)=\E f(x_0 + \a F)
= \int f(x_0 +\a z) p_F(z) dz
\]
and we apply our density estimate, so that
\[
\E f(X_\d)\leq \int  \frac{C f(x_0 +\a z)}{1+|z|^p} dz
\leq \frac{C}{|\det\a|} \int \frac{f(y)}{1+|x_0-y|_{A_\d(0,x_0)}^p} dy,
\]
in which we have used \eqref{p1}. Concerning $|\det \a|$, we recall \eqref{CB} and we obtain
\[
p_{X_\d}(y)\leq  \frac{1}{\d^{n-\frac{dim \langle \s(0,x_0) \rangle}{2}}} \frac{C}{1+|x_0-y|_{A_\d(0,x_0)}^p}.
\]
\epr

\br{2}
If Assumption \ref{assumption1bis} holds then the upper estimate in Theorem \ref{upper-th} is of exponential type:
there exists a constant $C\in \cal{C}$ such that for every $\d\leq 1$ and for every $y\in \R^n$
\[
p_{X_\d}(y)\leq  \frac{C}{\d^{n-\frac{dim \langle \s(0,x_0) \rangle}{2}}} \exp(- \frac{1}{C}|y-x_0|_{A_\d(0,x_0)}) .
\]
The proof is identical to the previous one except for the last part. In fact, looking at \eqref{boundF}, in this case the boundedness of the coefficients allows one to apply the exponential martingale inequality, so instead of \eqref{tails} we obtain the exponential bound $\PR(|F|>|y|/2)\leq C\exp(-|y|/C)$. This actually gives the proof of (3) in Theorem \ref{main-th}.
\er

\br{3}
In Theorem \ref{lower-th} the lower bound is centered at $x_0+\d b(x_0)$ but for the upper estimate in Theorem \ref{upper-th}, one can choose to center at $x_0$ or at $x_0+\d b(x_0)$. In fact, in this case we notice that
\[
|\d b(x_0)|_{A_\d(0,x_0)}\leq \frac{C'}{\d} |\d b(x_0)|\leq C'',
\]
so
\[
\frac{C_1}{1+|x_0-y|_{A_\d(0,x_0)}}\leq
\frac{C_2}{1+|x_0+\d b(x_0)-y|_{A_\d(0,x_0)}}\leq
\frac{C_3}{1+|x_0-y|_{A_\d(0,x_0)}},
\]
and the estimate of Theorem \ref{upper-th} can be equivalently written as
\[
p_{X_\d}(y)\leq  \frac{1}{\d^{n-\frac{\dim\langle \s(0,x_0) \rangle}{2}}
} \frac{C}{1+|y-x_0-\d b(x_0)|_{A_\d(0,x_0)}^p}.
\]
\er

\begin{remark}
Theorem \ref{upper-th} can be seen as an improvement of the upper bound in \cite{KusuokaStroock:85} in the sense that it precisely identifies the exponent
$n-\frac{\dim \langle \s(0,x_0) \rangle}{2}$, which accounts of the time-scale of the heat kernel when $\d$ goes to zero. This is evident when we consider the diagonal estimate $y=x_0$, and the same consideration holds when $y$ is close to $x_0$. When looking at the tails ($y$ far from $x_0$), it is not clear which of the two upper bounds is more accurate, unless we further specify the model.
\end{remark}

\appendix

\section{Proof of Lemma \ref{decZ-l}}\label{app-dec}

We prove the decomposition \eqref{decZ} in Lemma \ref{decZ-l}.
We recall $Z_t$ in \eqref{Decomp0}:
\[
Z_t=\sum_{i=1}^{d}a_{i}W_{t}^{i}+\sum_{i,j=1}^{d}a_{i,j}%
\int_{0}^{t}W_{s}^{i}\circ dW_{s}^{j}
\]
with $a_{i}=\sigma _{i}(0,x_0)$, $a_{i,j}=\partial _{\sigma _{i}}\sigma _{j}(0,x_0)$.
Setting $s_l=\frac{l}d\,\delta$, $l=1,\ldots,d$, we have
\begin{equation*}
Z_\delta=\sum_{l=1}^{d}Z(s_{l})-Z(s_{l-1})=\sum_{l=1}^{d}\left(
\sum_{i=1}^{d}a_{i}\Delta
_{l}^{i}+\sum_{i,j=1}^{d}a_{i,j}\int_{s_{l-1}}^{s_{l}}W_{s}^{i}\circ
dW_{s}^{j}\right).
\end{equation*}%
Recalling the quantities $\Delta_{l}^{j}$ and $\Delta _{l}^{i,j}$ in \eqref{delta1}, we write
\begin{equation*}
\int_{s_{l-1}}^{s_{l}}W_{s}^{i}\circ dW_{s}^{j}=W_{s_{l-1}}^{i}\Delta
_{l}^{j}+\Delta _{l}^{i,j}=(\sum_{p=1}^{l-1}\Delta _{p}^{i})\Delta
_{l}^{j}+\Delta _{l}^{i,j}.
\end{equation*}%
Then%
\begin{equation*}
Z_\delta=\sum_{l=1}^{d}\sum_{i=1}^{d}a_{i}\Delta
_{l}^{i}+\sum_{l=1}^{d}\sum_{i,j=1}^{d}a_{i,j}(\sum_{p=1}^{l-1}\Delta
_{p}^{i})\Delta _{l}^{j}+\sum_{l=1}^{d}\sum_{i,j=1}^{d}a_{i,j}\Delta
_{l}^{i,j}=:S_{1}+S_{2}+S_{3}.
\end{equation*}%
Notice first that%
\begin{equation*}
S_{1}=\sum_{l=1}^{d}a_{l}\Delta _{l}^{l}+\sum_{l=1}^{d}\sum_{i\neq
l}a_{i}\Delta _{l}^{i}.
\end{equation*}%
\bigskip We treat now $S_{3}.$ We will use the identities%
\begin{equation*}
\left\vert \Delta _{l}^{i}\right\vert ^{2}=2\Delta _{l}^{i,i}\quad \mbox{ and } \quad
\Delta _{l}^{i}\Delta _{l}^{j}=\Delta _{l}^{i,j}+\Delta _{l}^{j,i}.
\end{equation*}%
Then%
\begin{eqnarray*}
S_{3} &=&\sum_{l=1}^{d}\sum_{i=1}^{d}a_{i,i}\Delta
_{l}^{i,i}+\sum_{l=1}^{d}\sum_{i\neq j}a_{i,j}\Delta _{l}^{i,j} \\
&=&\sum_{l=1}^{d}\sum_{i=1}^{d}a_{i,i}\Delta
_{l}^{i,i}+\sum_{l=1}^{d}\sum_{i\neq l}a_{i,l}\Delta
_{l}^{i,l}+\sum_{l=1}^{d}\sum_{j\neq l}a_{l,j}\Delta
_{l}^{l,j}+\sum_{l=1}^{d}\sum_{i\neq j,i\neq lj\neq l}a_{i,j}\Delta
_{l}^{i,j} \\
&=&\frac{1}{2}\sum_{l=1}^{d}\sum_{i=1}^{d}a_{i,i}\left\vert \Delta
_{l}^{i}\right\vert ^{2}+\sum_{l=1}^{d}\sum_{i\neq l}a_{i,l}\Delta _{l}^{i,l}
\\
&&+\sum_{l=1}^{d}\sum_{j\neq l}a_{l,j}\left( \Delta _{l}^{j}\Delta
_{l}^{l}-\Delta _{l}^{j,l}\right) +\sum_{l=1}^{d}\sum_{i\neq j,i\neq l,j\neq
l}a_{i,j}\Delta _{l}^{i,j} \\
&=&\frac{1}{2}\sum_{i=1}^{d}a_{i,i}\left\vert \Delta _{i}^{i}\right\vert
^{2}+\frac{1}{2}\sum_{l=1}^{d}\sum_{i\neq l}^{d}a_{i,i}\left\vert \Delta
_{l}^{i}\right\vert ^{2}+\sum_{l=1}^{d}\sum_{i\neq l}(a_{i,l}-a_{l,i})\Delta
_{l}^{i,l} \\
&&+\sum_{l=1}^{d}\left( \sum_{j\neq l} a_{l,j}\Delta _{l}^{j}\right) \Delta
_{l}^{l}+\sum_{l=1}^{d}\sum_{i\neq j,i\neq l,\neq j\neq }a_{i,j}\Delta
_{l}^{i,j}.
\end{eqnarray*}%
We treat now $S_{2}.$ We want to emphasize the terms containing $\Delta
_{i}^{i}.$ We have
\begin{equation*}
S_{2}=
\sum_{l>p}^{d}
\sum_{i,j=1}^{d}a_{i,j}\Delta _{p}^{i}\Delta
_{l}^{j}=S_{2}^{\prime }+S_{2}^{\prime \prime }+S_{2}^{\prime \prime \prime
}+S_{2}^{iv}
\end{equation*}%
with
\begin{eqnarray*}
S_{2}^{\prime } &=&\sum_{l>p}^{d}a_{p,l}\Delta _{p}^{p}\Delta _{l}^{l},\quad
S_{2}^{\prime \prime }=\sum_{l>p}^{d}\sum_{j\neq l}a_{p,j}\Delta
_{p}^{p}\Delta _{l}^{j} \\
S_{2}^{\prime \prime \prime } &=&\sum_{l>p}^{d}\sum_{i\neq
p}^{d}a_{i,l}\Delta _{p}^{i}\Delta _{l}^{l},\quad
S_{2}^{iv}=\sum_{l>p}^{d}\sum_{i\neq p,j\neq l} a_{i,j}\Delta
_{p}^{i}\Delta _{l}^{j}.
\end{eqnarray*}%
We have
\begin{equation*}
S_{2}^{\prime \prime }=\sum_{p=1}^{d}\Delta _{p}^{p}\left(
\sum_{l=p+1}^{d}\sum_{j\neq l} a_{p,j}\Delta _{l}^{j}\right)
\end{equation*}%
and
\begin{equation*}
S_{2}^{\prime \prime \prime }=\sum_{l=1}^{d}\Delta _{l}^{l}\left(
\sum_{p=1}^{l-1}\sum_{i\neq p} a_{i,l}\Delta _{p}^{i}\right)
=\sum_{p=1}^{d}\Delta _{p}^{p}\left( \sum_{l=1}^{p-1}\sum_{j\neq
l} a_{j,p}\Delta _{l}^{j}\right)
\end{equation*}%
so that%
\begin{equation*}
S_{2}^{\prime \prime }+S_{2}^{\prime \prime \prime }=\sum_{p=1}^{d}\Delta
_{p}^{p}\left( \sum_{l=p+1}^{d}\sum_{j\neq l} a_{p,j}\Delta
_{l}^{j}+\sum_{l=1}^{p-1}\sum_{j\neq l} a_{j,p}\Delta _{l}^{j}\right) .
\end{equation*}
Finally%
\begin{eqnarray*}
Z_\delta &=&\sum_{l=1}^{d}a_{l}\Delta _{l}^{l}+\sum_{l=1}^{d}\sum_{i\neq
l}a_{i}\Delta _{l}^{i} \\
&&+\sum_{l>p}^{d}a_{p,l}\Delta _{p}^{p}\Delta _{l}^{l}+\sum_{p=1}^{d}\Delta
_{p}^{p}\left( \sum_{l>p}^{d}\sum_{j\neq l} a_{p,j}\Delta
_{l}^{j}+\sum_{p>l}^{d}\sum_{j\neq l} a_{j,p}\Delta _{l}^{j}\right) \\
&&+\sum_{l>p}^{d}\sum_{i\neq p,j\neq l} a_{i,j}\Delta _{p}^{i}\Delta
_{l}^{j}+\frac{1}{2}\sum_{i=1}^{d}a_{i,i}\left\vert \Delta
_{i}^{i}\right\vert ^{2}+\frac{1}{2}\sum_{l=1}^{d}\sum_{i\neq
l} a_{i,i}\left\vert \Delta _{l}^{i}\right\vert ^{2} \\
&&+\sum_{l=1}^{d}\sum_{i\neq l}(a_{i,l}-a_{l,i})\Delta
_{l}^{i,l}+\sum_{l=1}^{d}\left( \sum_{j\neq l}a_{l,j}\Delta _{l}^{j}\right)
\Delta _{l}^{l}+\sum_{l=1}^{d}\sum_{i\neq j,i\neq l,j\neq l}a_{i,j}\Delta
_{l}^{i,j}.
\end{eqnarray*}%
We want to compute the coefficient of $\Delta _{p}^{p}:$ this term appears
in $\sum_{p=1}^{d}\Delta _{p}^{p}(a_{p}+\varepsilon _{p})$, with
\[
\varepsilon _{p} =\sum_{l>p}^{d}\sum_{j\neq l} a_{p,j}\Delta
_{l}^{j}+\sum_{p>l}^{d}\sum_{j\neq l} a_{j,p}\Delta _{l}^{j}+\sum_{j\neq
p}a_{p,j}\Delta _{p}^{j}.
\]
We consider now $\Delta _{p}^{i,p}.$ It appears in
\begin{equation*}
\sum_{p=1}^{d}\sum_{i\neq p}(a_{i,p}-a_{p,i})\Delta _{p}^{i,p}
\end{equation*}%
The vector $a_{i,p}-a_{p,i}$ corresponds to the bracket $[\s_i,\s_p](0,x)$. Notice that for $l=l(i,p)$ when $i\neq p$, then $[\s_i,\s_p](0,x)=A_l(0,x)$, $A_l(0,x)$ being the $l$th column of $A(0,x)$. The other terms are%
\begin{eqnarray*}
&&\sum_{l=1}^{d}\sum_{i\neq l}a_{i}\Delta _{l}^{i}+\sum_{l>p}^{d}\sum_{i\neq
p,j\neq l} a_{i,j}\Delta _{p}^{i}\Delta _{l}^{j}+\frac{1}{2}%
\sum_{i=1}^{d}a_{i,i}\left\vert \Delta _{i}^{i}\right\vert ^{2}+\frac{1}{2}%
\sum_{l=1}^{d}\sum_{i\neq l} a_{i,i}\left\vert \Delta _{l}^{i}\right\vert
^{2} \\
&&+\sum_{l=1}^{d}\sum_{i\neq j,i\neq l,j\neq l}a_{i,j}\Delta
_{l}^{i,j}+\sum_{l=p+1}^{d}a_{p,l}\Delta _{p}^{p}\Delta _{l}^{l}.
\end{eqnarray*}%
We put everything together and \eqref{decZ} is proved.

\section{Support property}
\label{app-supp}

The aim of this section is the proof of the inequality in \eqref{Sup2}, which has been strongly used in Lemma \ref{lemma-Lambda}.

Let $B=(B^{1},...,B^{d-1})$ be
a standard Brownian motion. We consider the analogous of the covariance
matrix $Q_{i}(B)$ considered in Section \ref{lower-key}: we define a symmetric
square matrix of dimension $d\times d$ by
\begin{equation}\label{Q-app}
\begin{array}{l}
Q^{d,d} =1,\quad Q^{d,j}=Q^{j,d}=\int_{0}^{1}B_{s}^{j}ds,\quad j=1,...,d-1,
\smallskip\\
Q^{j,p} =Q^{p,j}=\int_{0}^{1}B_{s}^{j}B_{s}^{p}ds,\quad j,p=1,...,d-1
\end{array}
\end{equation}
and we denote by $\l_*(Q)$ (respectively by $\l^*(Q)) $ the lowest (respectively largest) eigenvalue of $Q$.


For a measurable function $g:[0,1]\rightarrow R^{d-1}$ we denote%
\begin{eqnarray*}
\alpha_{g}(\xi ) &=&\xi _{d}+\int_{0}^{1}\left\langle g_{s},\xi _{\ast
}\right\rangle ds,\quad \beta_{g}(\xi )=\int_{0}^{1}\left\langle g_{s},\xi
_{\ast }\right\rangle ^{2}ds-\left( \int_{0}^{1}\left\langle g_{s},\xi
_{\ast }\right\rangle ds\right) ^{2}\quad with \\
\xi &=&(\xi _{1},...,\xi _{d})\in \R^{d}\quad and\quad \xi _{\ast }=(\xi
_{1},...,\xi _{d-1}).
\end{eqnarray*}

We need the following two preliminary lemmas.

\begin{lemma}
\label{SUPORT1}With $g(s)=B_{s},s\in \lbrack 0,1]$ we have%
\begin{equation*}
\left\langle Q\xi ,\xi \right\rangle =\alpha_{B}^{2}(\xi )+\beta_{B}(\xi ).
\end{equation*}%
As a consequence, one has
\begin{equation*}
\lambda_*(Q)=\inf_{\left\vert \xi \right\vert
=1}(\alpha_{B}^{2}(\xi )+\beta_{B}(\xi ))\quad and\quad \l^*
(Q)\leq \sup_{\left\vert \xi \right\vert =1}(\alpha_{B}^{2}(\xi
)+\beta_{B}(\xi ))\leq \big(1+\sup_{t\leq 1}\left\vert B_{t}\right\vert\big)%
^2.
\end{equation*}%
Taking $\xi _{\ast }=0$ and $\xi _{d}=1$ we obtain $\left\langle Q\xi ,\xi
\right\rangle =1$ so that $\l_*(Q)\leq 1\leq \l^*(Q). $
\end{lemma}

\bpr By direct computation%
\begin{eqnarray*}
\left\langle Q\xi ,\xi \right\rangle &=&\xi _{d}^{2}+2\xi
_{d}\int_{0}^{1}\left\langle B_{s},\xi _{\ast }\right\rangle ds+\left(
\int_{0}^{1}\left\langle B_{s},\xi _{\ast }\right\rangle ds)\right) ^{2} \\
&&+\int_{0}^{1}\left\langle B_{s},\xi _{\ast }\right\rangle ^{2}ds-\left(
\int_{0}^{1}\left\langle B_{s},\xi _{\ast }\right\rangle ds\right) ^{2} \\
&=&\left( \xi _{d}+\int_{0}^{1}\left\langle B_{s},\xi _{\ast }\right\rangle
ds\right) ^{2}+\int_{0}^{1}\left\langle B_{s},\xi _{\ast }\right\rangle
^{2}ds-\left( \int_{0}^{1}\left\langle B_{s},\xi _{\ast }\right\rangle
ds\right) ^{2}.
\end{eqnarray*}%
The remaining statements follow straightforwardly.
\epr

\begin{proposition}
\label{SUPORT2}For each $p\geq 1$ one has%
\begin{equation}
\E(\left\vert \det Q\right\vert ^{-p})\leq C_{p,d}<\infty  \label{Sup1}
\end{equation}%
where $C_{p,d}$ is a constant depending on $p,d$ only.
\end{proposition}

\bpr By Lemma 7-29, pg 92 in \cite{[BGJ]}, for every $p\in
(0,\infty )$ one has%
\begin{equation*}
\frac{1}{\left\vert \det Q\right\vert ^{p}}\leq \frac{1}{\Gamma (p)}%
\int_{R^{d}}\left\vert \xi \right\vert ^{d(2p-1)}e^{-\left\langle Q\xi ,\xi
\right\rangle }d\xi .
\end{equation*}%
Let $\theta (\xi _{\ast }):=\int_{0}^{1}\left\langle B_{s},\xi _{\ast
}\right\rangle ds.$\ Using the previous lemma%
\begin{eqnarray*}
\int_{R^{d}}\left\vert \xi \right\vert ^{d(2p-1)}e^{-\left\langle Q\xi ,\xi
\right\rangle }d\xi &=&\int_{R^{d}}(\xi _{d}^{2}+\left\vert \xi _{\ast
}\right\vert ^{2})^{d(2p-1)/2}e^{-(\xi _{d}+\theta (\xi _{\ast
}))^{2}-\beta_{B}(\xi _{\ast })}d\xi \\
&\leq &C\int_{R^{d-1}}((1+\theta ^{2}(\xi _{\ast }))^{d(2p-1)/2}+\left\vert
\xi _{\ast }\right\vert ^{d(2p-1)})e^{-\beta_{B}(\xi _{\ast })}d\xi _{\ast }
\\
&\leq &C\int_{R^{d-1}}\sup_{t\leq 1}1\vee \left\vert B_{t}\right\vert
^{d(2p-1)}(1+\left\vert \xi _{\ast }\right\vert
^{d(2p-1)+1})e^{-\beta_{B}(\xi _{\ast })}d\xi _{\ast }.
\end{eqnarray*}%
We integrate and we use Schwartz inequality in order to obtain%
\begin{equation*}
\E\Big(\frac{1}{\left\vert \det Q\right\vert ^{p}}\Big)\leq
C+C\int_{\{\left\vert \xi _{\ast }\right\vert \geq 1\}}(\E((1+\left\vert \xi
_{\ast }\right\vert ^{d(2p-1)+1})^{2}e^{-2\beta_{B}(\xi _{\ast
})}))^{1/2}d\xi _{\ast }.
\end{equation*}%
For each fixed $\xi _{\ast }$ the process $b_{\xi _{\ast }}(t):=\left\vert
\xi _{\ast }\right\vert ^{-1}\left\langle B_{t},\xi _{\ast }\right\rangle $
is a standard Brownian motion and $\beta_{B}(\xi _{\ast })=\left\vert \xi
_{\ast }\right\vert ^{2}\int_{0}^{1}(b_{\xi _{\ast }}(t)-\int_{0}^{1}b_{\xi
_{\ast }}(s)ds)^{2}dt=:\left\vert \xi _{\ast }\right\vert ^{2}V_{\xi _{\ast
}}$\ where $V_{\xi _{\ast }}$\ is the variance of $b_{\xi _{\ast }}$\ with
respect to the time.\ Then it is proved in \cite{[DY]} (see (1.f), p. 183)
that
\begin{equation*}
\E(e^{-2\beta_{B}(\xi _{\ast })})=\E(e^{-2\left\vert \xi _{\ast }\right\vert
^{2}V_{\xi _{\ast }}})=\frac{2\left\vert \xi _{\ast }\right\vert ^{2}}{\sinh
2\left\vert \xi _{\ast }\right\vert ^{2}}.
\end{equation*}%
We insert this in the previous inequality and we obtain $\E(\left\vert \det
Q\right\vert ^{-p})<\infty .$
\epr

\medskip

We are now able to give the main result in this section. We define%
\begin{equation}
q(B)=\sum_{i=1}^{d-1}\left\vert B_{1}^{i}\right\vert +\sum_{j\neq
p}\left\vert \int_{0}^{1}B_{s}^{j}dB_{s}^{p}\right\vert  \label{Sup4}
\end{equation}%
and for $\varepsilon ,\rho >0$ we denote
\begin{equation}
\Upsilon_{\rho ,\varepsilon }(B)=\{\det Q\geq \varepsilon ^{\rho
},\sup_{t\leq 1}\left\vert B_{t}\right\vert \leq \varepsilon ^{-\rho
},q(B)\leq \varepsilon \}.  \label{Sup3}
\end{equation}

\begin{proposition}
\label{SUPORT3}There exist some universal constants $c_{\rho ,d},\varepsilon
_{\rho ,d}\in (0,1)$ (depending on $\rho $ and $d$ only) such that for every
$\varepsilon \in (0,\varepsilon _{\rho ,d})$ one has%
\begin{equation}
\PR(\Upsilon_{\rho ,\varepsilon }(B))\geq c_{\rho ,d}\times \varepsilon ^{%
\frac{1}{2}d(d+1)}.  \label{Sup2}
\end{equation}
\end{proposition}

\bpr Using the previous proposition and Chebyshev's inequality we
get
\begin{equation*}
\PR(\det Q<\varepsilon ^{\rho })\leq \varepsilon ^{p\rho }\E\left\vert \det
Q\right\vert ^{-p}\leq C_{p,d}\varepsilon ^{p\rho } \quad\mbox{and}\quad
\PR(\sup_{t\leq 1}\left\vert B_{t}\right\vert >\varepsilon ^{-\rho })\leq \exp
(-\frac{1}{C\varepsilon ^{2\rho }}).
\end{equation*}
Let $q^{\prime }(B)=\sum_{i=1}^{d-1}\left\vert B_{1}^{i}\right\vert
+\sum_{j<p}\left\vert \int_{0}^{1}B_{s}^{j}dB_{s}^{p}\right\vert .$ Since $%
\left\vert \int_{0}^{1}B_{s}^{j}dB_{s}^{p}\right\vert \leq \left\vert
B_{1}^{j}\right\vert \left\vert B_{1}^{p}\right\vert +\left\vert
\int_{0}^{1}B_{s}^{p}dB_{s}^{j}\right\vert $ we have $q(B)\leq 2q^{\prime
}(B)+q^{\prime}(B)^2$ so that $\{q^{\prime }(B)\leq \frac{1}{3}\varepsilon
\}\subset \{q(B)\leq \varepsilon \}.$ We will now use the following fact:
consider the diffusion process $X=(X^{i},X^{j,p},$ $i=1,...,d,1\leq j<p\leq
d)$ solution of the equation $%
dX_{t}^{i}=dB_{t}^{i},dX_{t}^{j,p}=X_{t}^{j}dB_{t}^{p}.$ The strong H\"{o}rmander condition holds for this process and the support of the law of $%
X_{1} $ is the whole space. So the law of $X_{1}$ is absolutely continuous
with respect to the Lebesgue measure and has a continuous and strictly
positive density $p.$ This result is well known (see for example \cite{KusuokaStroock:87}
or \cite{BallyCaramellino:12}). We denote $c_{d}:=\inf_{\left\vert x\right\vert \leq
1}p(x)>0$ and this is a constant which depends on $d$ only. Then, by
observing that $q^{\prime}(B)\leq \sqrt{m} \,|X_1|$, where $m=\frac{1}{2}%
d(d+1)$ is the dimension of the diffusion $X$, we get
\begin{equation*}
\PR(q(B)\leq \varepsilon )\geq \P\Big(q^{\prime }(B)\leq \frac{\varepsilon }{3}%
\Big)\geq \P\Big(\left\vert X_{1}\right\vert \leq \frac{\varepsilon }{3\sqrt m%
}\Big)\geq \frac{\varepsilon ^{m}}{(3\sqrt m)^{m}}\times \bar c_{d},
\end{equation*}%
with $\bar c_d>0$. So finally we obtain
\begin{equation*}
\PR(\Upsilon_{\rho ,\varepsilon }(B))\geq \bar c_{d}\varepsilon ^{\frac{1}{2}%
d(d+1)}-C_{p,d}\varepsilon ^{p\rho }-\exp (-\frac{1}{C\varepsilon ^{2\rho }}%
).
\end{equation*}%
Choosing $p>\frac{1}{2\rho }d(d+1)$ and $\varepsilon $ small we obtain our
inequality.
\epr

\section{Density estimates via local inversion}
\label{sectioninversefunction}

In this section we see how to use the inverse function theorem to transfer a known estimate for a Gaussian random variable to its image via a certain function $\eta$. For a standard version of the inverse function theorem see \cite{rudin-principles}.

We consider $\Phi(\theta)=\theta+\eta(\theta)$, for a three times differentiable function $\eta :\R^m \rightarrow \R^m $. Define
\be{ci}
c_2(\eta) =\max_{i,j=1,..,m}\sup_{|x|\leq 1}|\partial^2_{ij}\eta(x)|,\quad
c_3(\eta) =\max_{i,j,k=1,..,m}\sup_{|x|\leq 1}|\partial^3_{ijk}\eta(x)|,
\ee
and
\be{defheta}
h_\eta=\frac{1}{16 m^2 (c_2(\eta)+\sqrt{c_3(\eta)})}
\ee

\begin{lemma}\label{invfun1}
Take $h_\eta$ as above. If the function $\eta$ is such that
\begin{equation*}
\eta\in C^3(\R^m,\R^m),\quad \eta(0)=0,\quad \nabla\eta(0)\leq \frac{1}{2},
\end{equation*}
then there exists a neighborhood of $0$, that we denote with $V_{h_\eta} \subset B(0,2 h_\eta)$, such that $\Phi: V_{h_\eta}\rightarrow B\left(0,\frac{1}{2} h_\eta\right)$ is a diffeomorphism. In particular, if we denote with $\Phi^{-1}$ the local inverse of $\Phi$, we have
\begin{equation*}
\Phi^{-1}: B\left(0,\frac{1}{2} h_\eta\right) \rightarrow B\left(0,2 h_\eta\right),
\end{equation*}
and we have this quantitative estimate:
\begin{equation}\label{estinv}
\forall y \in B\left(0,\frac{1}{2} h_\eta\right),\quad \frac{1}{4} |\Phi^{-1}(y)|\leq |y| \leq 4|\Phi^{-1}(y)|.
\end{equation}
\end{lemma}
\br{trw}
Here we write $\Phi^{-1}$ for the inverse of the restriction of $\Phi$ to $V_{h_\eta}$, what is called a \emph{local} inverse.
\er

\begin{proof}
We have
\begin{equation*}
\nabla\Phi(0) = \mathrm{Id}+\nabla\eta(0).
\end{equation*}
So
\begin{equation*}
|\nabla\Phi(0) x|^2\geq \frac{1}{2}|x|^2 - |\nabla\eta(0)x|^2 \geq \frac{1}{2}|x|^2 - \frac{1}{4}|x|^2=\frac{1}{4}|x|^2.
\end{equation*}
and
\begin{equation*}
|\nabla\Phi(0) x|^2 \leq 2|x|^2 + 2|\nabla\eta(0)x|^2
\leq 2|x|^2 + \frac{1}{2}|x|^2\leq \frac{5}{2}|x|^2.
\end{equation*}
Therefore
\begin{equation*}
\frac{1}{2}|x| \leq |\nabla\Phi(0) x| \leq \sqrt{3}|x|
\end{equation*}
This implies $\Phi(0)$ is invertible locally around $0$, and the local inverse differentiable, using the classical inverse function theorem. We now look now at the image of the inverse, and at the estimates \eqref{estinv}.
We develop $\eta$ around $0$, writing $\nabla^2 \eta(x)[u,v]$ to denote $\nabla^2\eta(x)$ computed in $u$ and $v$.
\[
\eta(\theta)=\nabla\eta(0)\theta+\int_0^1(1-t)\nabla^2\eta (t\theta)[\theta,\theta] dt.
\]
Fix $y\in \R^m$. Suppose $\Phi(\theta)=y$. Then
\begin{align*}
\theta &=(\nabla\Phi(0))^{-1} \nabla\Phi(0) \theta\\
&=(\nabla\Phi(0) )^{-1}
(\theta+\nabla\eta(0) \theta ) \\
&=(\nabla\Phi(0) )^{-1}
\left(\theta+ \eta(\theta) -\int_0^1(1-t)\nabla^2\eta(t\theta)[\theta,\theta] dt  \right) \\
&=(\nabla\Phi(0) )^{-1}
\left( y -\int_0^1(1-t) \nabla^2\eta(t\theta)[\theta,\theta] dt \right).
\end{align*}
We define
$$
U_y(\theta)=\left( y -\int_0^1(1-t) \nabla^2\eta(t\theta)[\theta,\theta] dt \right),
$$
so that $\theta$ can be seen as a fixed point for $U_y$.
Recall that $|\frac{1}{2}x|\leq |\nabla\Phi(0)x |$.
\begin{align*}
|U_y(\theta_1)-U_y(\theta_2)|&=
\left| (\nabla\Phi(0) )^{-1}
\left( \int_0^1(1-t) (\nabla^2\eta(t\theta_2)[\theta_2,\theta_2]-\nabla^2\eta(t\theta_1)[\theta_1,\theta_1]) dt \right) \right| \\
&\leq 2 \left|
\int_0^1(1-t) (\nabla^2\eta(t\theta_2)[\theta_2,\theta_2]-\nabla^2\eta(t\theta_1)[\theta_1,\theta_1]) dt \right| \\
&\leq 2
\int_0^1(1-t) (|\nabla^2\eta(t\theta_1)[\theta_1,\theta_1-\theta_2]|+|\nabla^2\eta (t\theta_1)[\theta_1-\theta_2,\theta_2]| \\
&+| \nabla^2\eta(t\theta_1)[\theta_2,\theta_2]-\nabla^2\eta (t\theta_2)[\theta_2,\theta_2] | ) dt.
\end{align*}
Now, form \eqref{defheta}, for $\th_1,\th_2 \in B(0, h_\eta)$
\begin{align*}
|\nabla^2\eta(t\theta_1)[\theta_1,\theta_1-\theta_2]|\leq m^2 c_2(\eta) h_\eta |\theta_1-\theta_2|\leq \frac{1}{16}|\th_1-\th_2|,
\end{align*}
and
\begin{align*}
|\nabla^2\eta(t\theta_1)[\theta_2,\theta_2]-\nabla^2\eta(t\theta_2)[\theta_2,\theta_2] | \leq m^3 c_3(\eta) |\theta_1-\theta_2| h_\eta^2\leq \frac{1}{256}|\th_1-\th_2|,
\end{align*}
and therefore
\begin{equation}\label{contr}
|U_y(\theta_1)-U_y(\theta_2)| \leq
\frac{1}{4} |\theta_1-\theta_2|.
\end{equation}

For $y\in B(0,\frac{1}{2}h_\eta)$ and $\theta\in B(0,2h_\eta)$ this implies
$$
|U_y(\theta)|\leq |U_y(\theta)-U_y(0)|+|U_y(0)| \leq \frac{1}{4}|\theta|+2y \leq 2 h_\eta
$$
Define now the sequence
$$
\theta_0=0,\quad \theta_{k+1}=U_y(\theta_k).
$$
We know that $\theta_k\in B(0,2 h_\eta)$ for any $k\in \N$, and therefore inequality \eqref{contr} implies
\begin{equation*}
|U_y(\theta_k)-U_y(\theta_{k+1})| \leq
\frac{1}{4} |\theta_k-\theta_{k+1}|.
\end{equation*}
The Banach fixed-point theorem tells us that $\theta_k$ converges to the unique solution of $\theta=U_y(\theta)$, which is $\theta=\Phi^{-1}(y)$, and $\theta\in B(0,2 h_\eta)$. So it is possible to define the local inverse $\Phi^{-1}$ on $B\left(0,\frac{1}{2} h_\eta\right)$, and
$$
V_{h_\eta}:=\Phi^{-1}B\left(0,\frac{1}{2} h_\eta\right)\subset B(0,2 h_\eta).
$$
Now, for $y\in B(0,\frac{1}{2}h_\eta)$, let $\theta=\Phi^{-1}(y)$ and the following inequalities hold
\begin{align*}
|\theta|&=|U_y(\theta)|\leq \frac{1}{2}\theta+2|y|
&\Rightarrow &
|\theta|\leq 4|y| \\
|\theta|&=U_y(\theta)\geq|U_y(0)|-|U_y(\theta)-U_y(0)|\geq \frac{1}{2}|y| -\frac{1}{2}|\theta|
&\Rightarrow &|\theta|\geq \frac{1}{4}|y|.
\end{align*}
\end{proof}

Let now $\Theta$ be a $m$-dimensional centered Gaussian variable with covariance matrix $Q$. Denote by $\underline{\lambda}$ and $\overline{\lambda}$ the lowest and the largest eigenvalues of $Q$. Keeping in mind the setting of the last subsection, we also introduce the notation
\be{c*}
c_*(\eta,h)=\sup_{|x|\leq 2h} \max_{i,j} |\partial_i\eta^j(x)|
\ee
for $h>0$. Recall we are supposing $\eta\in C^3(\R^m,\R^m)$ and $\eta(0)=0$.

Take $r>0$ such that
\be{hpimpl}
c_*(\eta,16 r)\leq \frac{1}{2m} \sqrt{\frac{\underline{\l}}{\overline{\l}}},\quad \quad
r\leq h_\eta =\frac{1}{16 m^2 (c_2(\eta)+\sqrt{c_3(\eta)})}.
\ee
We take a localizing function as in \eqref{defU}:
\be{hpimplbis}
U=\prod_{i=1}^m \psi_{r}(\Theta_i).
\ee
\bl{invfun2}
Let $Q$ be non degenerate. Let $r$ such that \eqref{hpimpl} holds and set $U$ as in \eqref{hpimplbis}. Then the density $p_{G,U}$ of
\[
G:=\Phi(\Theta)=\Theta+\eta(\Theta)
\]
under $\PR_U$ has the following bounds on $B(0,r)$:
\be{invest}
\frac{1}{C \det Q^{1/2}
}\exp\left(-\frac{C}{\underline{\l}}|z|^2 \right) \leq
p_{G,U}(z) \leq \frac{C}{\det Q^{1/2} }\exp\left(-\frac{1}{C \overline{\l}}|z|^2 \right)
\ee
\el
\begin{proof}
For a general nonnegative, measurable function $f:\R^m\rightarrow \R$ with support included in $B(0,4r)$, we compute $\E(f(G) 1_{\{\Theta\in \Phi^{-1}B(0,4r)\}})$. Here $\Phi^{-1}$ is the local diffeomorphism of the inverse function theorem. After the multiplication with the characteristic function, on the support of the random variable that we are averaging, $\Phi$ is a diffeomorphism and the first equality holds. The second follows from the change of variable suggested by Lemma \ref{invfun1} for $G=\Phi(\Th)$
\begin{align*}
&\E\left(f(G) 1_{\{\Theta\in \Phi^{-1}B(0,4r)\}}\right)=\\
&\quad=\int_{\Phi^{-1}(B(0,4r))} f(\Phi(\theta))\frac{1}{(2\pi)^{m/2}\det Q^{1/2}}\exp\left(-\frac{1}{2}\langle Q^{-1}\theta,\theta\rangle \right) d\theta\\
&\quad=\int_{B(0,4r)} f(z) \bar{p}_G(z) dz,
\end{align*}
where for $z\in B(0,4r)$
$$
\bar{p}_G(z)=
\frac{1}{(2\pi)^{m/2}\det Q^{1/2} |\det \nabla\Phi(\Phi^{-1}(z))|
}\exp\left(-\frac{1}{2}\langle Q^{-1}\Phi^{-1}(z),\Phi^{-1}(z)\rangle \right).
$$
Again from Lemma \ref{invfun1}, since $4r\leq \frac{h_\eta}{2}$, we have $z\in B(0,4r)\Rightarrow \theta \in B(0,16r)$. Using $c_*(\eta,16r)\leq \frac{1}{2m} \sqrt{\frac{\underline{\l}}{\overline{\l}}}$,
$$
\frac{1}{2}|x|^2\leq(1-m\,c_*(\eta,h_\eta))|x|^2
\leq |\langle\nabla\Phi(\theta)x,x \rangle|
\leq(1+m\,c_*(\eta,h_\eta))|x|^2\leq 2|x|^2.
$$
Therefore if $z\in B(0,4r)$
$$
2^{-m}\leq |\det \Phi(\Phi^{-1}(z))|\leq 2^m.
$$
Moreover, using Lemma \ref{invfun1}
\begin{align*}
&\langle Q^{-1} \Phi^{-1}(z),\Phi^{-1}(z)\rangle\leq \frac{1}{\underline{\lambda}}|\Phi^{-1}(z)|^2\leq
\frac{16}{\underline{\lambda}}|z|^2,\\
&\langle Q^{-1} \Phi^{-1}(z),\Phi^{-1}(z)\rangle \geq
\frac{1}{\overline{\lambda}}|\Phi^{-1}(z)|^2\geq
\frac{1}{16\overline{\lambda}}|z|^2.
\end{align*}
Therefore
\begin{align*}
\frac{1}{(8\pi)^{m/2}\det Q^{1/2}
}\exp\left(-\frac{8}{\underline{\lambda}}|z|^2 \right)
\leq \bar{p}_G(z) \leq
\frac{2^{m/2}}{\pi^{m/2}\det Q^{1/2}
}\exp\left(-\frac{1}{32\overline{\lambda}}|z|^2 \right).
\end{align*}
Now we define, as in \eqref{defU} the localization variables
\[
U_1=\prod_{i=1}^m \psi_{16r}(\Theta_i),\quad U_2=\prod_{i=1}^m \psi_{r}(\Theta_i).
\]
Notice that
\[
U_2\leq 1_{\{\Theta\in \Phi^{-1}B(0,4r)\}}\leq U_1,
\]
so that we have
\begin{align*}
\E\left(f(G) U_2\right)\leq
\E\left(f(G) 1_{\{\Theta\in \Phi^{-1}B(0,4r)\}}\right)
\leq
\E\left(f(G) U_1\right).
\end{align*}
The following bounds for the local densities follow:
\begin{align*}
p_{G,U_1}(z)&\geq\frac{1}{(8\pi)^{m/2}\det Q^{1/2}
}\exp\left(-\frac{8}{\underline{\lambda}}|z|^2 \right),
\\
p_{G,U_2}(z)&\leq \frac{2^{m/2}}{\pi^{m/2}\det Q^{1/2}
}\exp\left(-\frac{1}{32\overline{\lambda}}|z|^2 \right).
\end{align*}
$U_1 \geq U = U_2$, so for the localization via $U$ both bounds hold.
\end{proof}

\section{Estimates of the distance between localized densities}\label{app-mall}

\subsection{Elements of Malliavin calculus}
We recall some basic notions in Malliavin calculus. Our main reference is \cite{Nualart:06}. We consider a probability space $(\Omega,\mathcal{F},\PR)$ and a Brownian motion $W=(W^1_t,...,W^d_t)_{t\geq 0}$ and the filtration $(\mathcal{F}_t)_{t \geq 0}$ generated by $W$. For fixed $T>0$, we denote by $\mathcal{H}$ the Hilbert space $L^2([0,T],\R^d)$. For $h\in \mathcal{H}$ we introduce this notation for the It\^{o} integral of $h$: $W(h)=\sum_{j=1}^d \int_0^T h^j(s)dW_s^j$.

We denote by $C_p^\infty (\R^n)$ the set of all infinitely continuously differentiable functions $f : \R^n \rightarrow \R$ such that $f$ and all of its partial derivatives have polynomial growth. We also denote by $\mathcal{S}$ the class of simple random variables of the form
\[
F=f(W(h_1),...,W(h_n)),
\]
for some $f\in C_p^\infty(\R^n)$, $h_1,...,h_n$ in $\H$, $n\geq 1$.
The Malliavin derivative of $F\in \mathcal{S}$ is the $\H$ valued random variable given by
\be{defdermal}
DF=(DF^1,\dots,DF^d)^T=\sum_{i=1}^n \frac{\partial f}{\partial x_i}(W(h_1),...,W(h_n))h_i.
\ee
We introduce the Sobolev norm of $F$:
$$
\|F\|_{1,p}=[\E|F|^p+\E |DF|^p]^\frac{1}{p}
$$
where
$$
|DF|=\left(\int_0^T |D_s F|^2 ds\right)^\frac{1}{2}.
$$
It is possible to prove that $D$ is a closable operator and take the extension of $D$ in the standard way.
We can now define in the obvious way $DF$ for any $F$ in the closure of $\mathcal{S}$ with respect to this norm. Therefore, the domain of $D$ will be the closure of $\mathcal{S}$.

The higher order derivative of $F$ is obtained by iteration. For any $k \in \N$, for a multi-index $\alpha=(\alpha_1,...,\alpha_k)\in\{1,...,d\}^k$ and $(s_1,...,s_k)\in [0,T]^k$, we can define
$$
D^\alpha_{s_1,...,s_k} F := D^{\alpha_1}_{s_1} ... D^{\alpha_k}_{s_k} F.
$$
We denote by $|\alpha|=k$ the length of the multi-index. Remark that $D^\alpha_{s_1,...,s_k} F$, is a random variable with values in $\H^{\otimes k}$, and so we define its Sobolev norm as
$$
\|F\|_{k,p}=[\E|F|^p+\sum_{j=1}^k \E |D^{(j)} F|^p]^\frac{1}{p}
$$
where
$$
|D^{(j)} F|=\left(\sum_{|\alpha|=j}\int_{[0,T]^j}|D^\alpha_{s_1,...,s_j}F|^2 d s_1 ... d s_j\right)^{1/2}.
$$
The extension to the closure of $\mathcal{S}$ with respect to this norm is analogous to the first order derivative. Notice that with this notation $|DF|=|D^{(1)}F|$. Also notice that $D^{(j)}$ means "derivative of order $j$" and $D^j$ means "derivative with respect to $W^j$".

We denote by $\DD^{k,p}$ the space of the random variables which are $k$ times differentiable in the Malliavin sense in $L^p$, and $\DD^{k,\infty}=\bigcap_{p=1}^\infty \DD^{k,p}$. As usual, we also denote by $L$ the Ornstein-Uhlenbeck operator, i.e. $L=-\d\circ D$, where $\d$ is the adjoint operator of $D$.

We consider random vector $F=(F_1,...,F_n)$ in the domain of $D$. We define its \emph{Malliavin covariance matrix} as follows:
\[
\g_F^{i,j}=\langle D F_i,D F_j\rangle_\mathcal{H} = \sum_{k=1}^d \int_0^T D^k_s F_i\times D^k_s F_j ds.
\]

\subsection{Localization and density estimates}\label{local}
The following notion of localization is introduced in \cite{BallyCaramellino:12}. Consider a random variable $U\in[0,1]$ and denote
$$
d\PR_U=Ud\PR.
$$
$\PR_U$ is a non-negative measure (not a probability measure, in general). We also set $\E_U$ the expectation (integral) w.r.t. $\PR_U$, and denote
\begin{align*}
\| F \|_{p,U}^p&=\E_U(|F|^p)=\E(|F|^p U)\\
\| F \|_{k,p,U}^p&=\| F \|_{p,U}^p+\sum_{j=1}^k \E_U(|D^{(j)}F|^p).
\end{align*}
We assume that $U\in \DD^{1,\infty}$ and for every $p\geq 1$
\be{defmU}
m_U(p):=1+ \E_U | D \ln U |^{p} < \infty.
\ee
The specific localizing function we will use is the following. Consider the function depending on a parameter $a>0$:
\[
\psi_a(x)=1_{|x|\leq a}+\exp\left( 1-\frac{a^2}{a^2-(x-a)^2} \right)1_{a<|x|<2a}.
\]
For $\Theta_i\in \DD^{2,\infty}$ and $a_i>0$, $i=1\dots,n$ we define the localization variable:
\be{defU}
U=\prod_{i=1}^n \psi_{a_i}(\Theta_i)
\ee
For this choice of $U$ we have that for any $p,k\in \N_0$
\be{normlocalization}
m_U(p)
\leq
C \frac{\|\Theta\|_{1,p}^p}{|a|}
\ee
The proof of \eqref{normlocalization} follows from standard computations and inequality
\be{boundpsi}
\sup_x |(\ln \psi_a) (x)|^p \psi_a(x) \leq
\frac{C}{a^{p}}
\ee
In the following proposition we state the general lower and upper bound that we use in our density estimate. These results come from \cite{BallyCaramellino:12} and  \cite{BC14}.

\bp{generalbounds}
Let $F\in (\DD^{2,\infty})^d$.
\ben
\item
Suppose that for every $p\in \N: \E_U |\l_*(\g_F)|^{-p}<\infty$, $U\in \DD^{1,\infty}$ and $m_U(p)< \infty$.
Let $G \in (\DD^{2,\infty})^d$ such that for every $p\in \N$
\[
\E_U |\l_*(\g_G)|^{-p}<\infty.
\]
Then for every $p>d$
\be{generallowerbound}
\begin{split}
&p_{F,U}(y)\\
&\geq
p_{G,U}(y)- C m_U(p)^b \max\big\{1, (\E_U |\l_*(\g_G)|^{-p}) ^b
(\| F \|_{2,p,U}+\| G \|_{2,p,U}) \big\}
 \|F-G\|_{2,p,U}
\end{split}
\ee
where $C, b$ are constants depending only on $d,p$ and $m_U(p)$ is given by \eqref{defmU}.
\item
Assume $\E |\l_*(\g_F)|^{-p}<\infty,\,\forall p$. Then $\exists C, p, b$ constants depending only on the dimension $d$ such that
\be{generalupperbound}
|p_F(y)|\leq C \max\{1,\E |\l_*(\g_F)|^{-p} \|F\|_{2,p}  \} \PR(|F-y|<2)^b
\ee
\een
\ep

\bpr{}
\ben
\item
The lower bound \eqref{generallowerbound} for $p_{F,U}$ is a version of Proposition 2.5. in
\cite{BallyCaramellino:12} with the lowest eigenvalue instead of the determinant.
\item
The upper bound \eqref{generalupperbound} for $p_{F}$ is a version of Theorem 2.14, point A., in \cite{BC14}.  We take therein $q=0$, so there is no derivative, and $\Th=1$, that means that we are not localizing.
\een
\epr

\bibliographystyle{plain}
\bibliography{bibliografia}

\end{document}